\documentclass[11pt]{article}
\usepackage[usenames,dvipsnames]{color}

\usepackage{amsmath,amssymb}
\usepackage{mathrsfs}
\usepackage{color,graphics}
\usepackage{graphicx}
\usepackage{amsthm}
\usepackage{natbib}
\newcommand{\todo}[1]{{\color{magenta}#1}}   % TO DO'S
\numberwithin{equation}{section}
\theoremstyle{plain}
\newtheorem{thm}{Theorem}[section]
\newtheorem{prop}{Proposition}[section]
\newtheorem{rmk}{Remark}[section]

\newtheorem{cor}{Corollary}[section]

\newcommand{\red}{\textcolor{red}}

\newcommand{\N}{\mathbb{N}}
\newcommand{\R}{\mathbb{R}}

\newcommand{\hatm}{\hat{m}}

\newcommand{\vir}[1]{``#1''}

\newcommand{\pg}[1]{\left\{#1\right\}}
\newcommand{\pq}[1]{\left[#1\right]}
\newcommand{\pt}[1]{\left(#1\right)}
\newcommand{\abs}[1]{\left\vert#1\right\vert}
\newcommand{\bs}[1]{\mathbf{#1}}

\newcommand{\ptn}{p_\theta^{(n)}}
\newcommand{\Ptn}{P_\theta^{(n)}}
\newcommand{\pto}{p_{\theta_0}^{(n)}}
\newcommand{\Pto}{P_{\theta_0}^{(n)}}
\newcommand{\ptzo}{p_{(\tau_0,\,\zeta_0)}^{(n)}}

\newcommand{\ptnz}{p_{(\tau_n,\,\zeta)}^{(n)}}
\newcommand{\pyz}{p^{(n)}_{(Y_n^{(\omega)}(\omega'),\,\zeta)}}

\newcommand{\MMLE}{marginal MLE}%{Maximum Marginal Likelihood Estimator}
\newcommand{\argmax}{\operatornamewithlimits{argmax}}

\newcommand{\data}{x_{1:n}} %x_1, \ldots, x_n}
\newcommand{\Data}{X_{1:n}} %X_1, \ldots, X_n}
\newcommand{\hatl}{\hat{\lambda}}

\usepackage[T1]{fontenc}

\begin{document}

\allowdisplaybreaks

\begin{center}
{\bf \Large Bayes and empirical Bayes: do they merge?}

\bigskip

Sonia Petrone$^{1}$, Judith Rousseau$^{2}$ \& Catia Scricciolo$^{3}$

\bigskip

{\it $^{1}$ Bocconi University, sonia.petrone@unibocconi.it

$^{2}$ CREST-ENSAE and Universit\'e Paris Dauphine, rousseau@ceremade.dauphine.fr

$^{3}$ Bocconi University and CREST-ENSAE, catia.scricciolo@unibocconi.it}

\end{center}

\bigskip

\begin{abstract}
Bayesian inference is attractive for its coherence and good frequentist properties. However, it is a common experience that eliciting a honest prior may be difficult and, in practice, people often take an {\em empirical Bayes} approach, plugging empirical estimates of the prior hyperparameters into the posterior distribution. Even if not rigorously justified, the underlying idea is that, when the sample size is large, empirical Bayes leads to \vir{similar} inferential answers. Yet, precise mathematical results seem to be missing. In this work, we give a more rigorous justification in terms of merging of Bayes and empirical Bayes posterior distributions. We consider two notions of merging: Bayesian weak merging and frequentist merging in total variation. Since weak merging is related to consistency, we provide sufficient conditions for consistency of empirical Bayes posteriors. Also, we show that, under regularity conditions, the empirical Bayes procedure asymptotically selects the value of the hyperparameter for which the prior mostly favors the \vir{truth}. Examples include empirical Bayes density estimation with Dirichlet process mixtures.
\end{abstract}

\bigskip


\bigskip

\section{Introduction and motivation} \label{sec:intro}
The Bayesian approach to inference is appealing in treating uncertainty probabilistically through conditional distributions. If $(X_1,\,\ldots,\,X_n)|\theta$ have joint density $\ptn$
and $\theta$ has prior density $\pi(\theta|\lambda)$,
then the information on $\theta$, given the data, is expressed through the conditional, or posterior, density $\pi(\theta|\lambda,\, x_1,\,\ldots,\,x_n) \propto \ptn(x_1,\,\ldots,\,x_n)\pi(\theta|\lambda)$.
Despite Bayes procedures are increasingly popular, it is a common experience that expressing honest prior information can be difficult and, in practice, one is tempted to use some estimate $\hatl_n\equiv\hatl_n(x_1,\,\ldots,\,x_n)$ of the prior hyperparameter $\lambda$ and a posterior distribution $\pi(\cdot|\hatl_n,\,x_1,\,\ldots,\,x_n)$. This mixed approach is usually referred to as {\em empirical Bayes} in the literature (see, \emph{e.g.}, \cite{LC98}) and we shall also refer to it as empirical Bayes (EB) in this paper. The underlying idea is that, when $n$ is large, EB should reasonably lead to inferential results similar to those of any Bayes procedure. Thus, an empirical Bayesian would achieve the goal of inference without completely specifying a prior. From a Bayesian point of view, when there is uncertainty on $\lambda$, a cleaner solution than EB would be to put a hyperprior on $\lambda$. However, using a hierarchical prior is often computationally heavy, thus EB appears as an attractive alternative that is expected to be \vir{better} than fixing a \vir{wrong} value of the hyperparameter. Recently, \cite{scott:berger:10} showed that, in the special case of variable selection in regression models, EB procedures could have very strange features and rather pathological behaviours. They showed that if a Bayesian approach is based on a prior on models of the form
$$\Pi(M_{\boldsymbol{\gamma}}|p)\propto p^{k_{\boldsymbol{\gamma}}}( 1 - p)^{m - k_{\boldsymbol{\gamma}}}, \qquad p\in (0,\,1),$$
where $M_{\boldsymbol{\gamma}}$ denotes the model associated to the vector of inclusions ${\boldsymbol{\gamma}}\in \{0,\,1\}^m$, $m$ being the total number of candidate covariates, and $k_{\boldsymbol{\gamma}}$ corresponds to the number of covariates used in the model ${\boldsymbol{\gamma}}$, then the maximum marginal likelihood estimator (\MMLE) for $p$ can be equal to $0$ or $1$ with positive probability. This corresponds to a very degenerate EB prior on the distribution of models. This might still lead to interesting  pointwise estimates of the model or of the whole parameter, however, in terms of distribution (posterior) is far from being satisfactory.

In this paper, we try to shed light on such phenomena by answering the following question:
is it actually true that an EB procedure is asymptotically close to a regular Bayes procedure?
Of course, it depends on what \textit{close} here means. So, first we consider a weak notion of closeness which mainly corresponds to consistency and is still satisfied in the above \textit{pathological} example. Then, we give a generic explanation of the above phenomenon, describing when and why \MMLE\, EB procedures will be pathological or, on the contrary, when and why they will have some quality of optimality.

We start by formalizing the asymptotic comparison in terms of {\em merging} of Bayes and EB procedures.
We consider two notions of merging. First, we study {\em Bayesian weak merging} in the sense of \cite{DiaconisFreedman1986}, which means that any Bayesian is sure that her/his posterior and the EB posterior will eventually be close, in the sense of weak convergence.
It thus appears as a minimal requirement in order for the EB procedure to be reasonable from a Bayesian viewpoint. Then, we study {\em frequentist strong merging} in the sense of \cite{ghoshRamamoorthi2003},
which means that the Bayes and EB posterior distributions will be eventually close, in the total variation distance, $P_{0}^\infty$-almost surely, where $P_0^\infty\equiv P_{\theta_0}^\infty$ denotes the probability law of $(X_i)_{i\geq1}$ under $\theta_0$, the \emph{true} value, in the frequentist sense, of the parameter $\theta$, usually referred to as the \textit{true} distribution. It is worth noting that, when strong merging holds, if Bernstein von-Mises holds in the $L_1$-sense for the Bayes posterior, then it also holds for the EB posterior. We prove that, under some conditions, Bayesian and frequentist merging of Bayes and EB procedures hold. However, these results require conditions, possibly more restrictive than expected. Thus, a further contribution of this work is to identify some typical situations wherein merging fails, which, despite appearances, are not surprising. These results have therefore the practical interest of characterizing, at least in the parametric case, those families of priors to be used with regard to EB procedures and those to be avoided, in particular if one is not merely interested in point estimation, but in some more general characteristics of the posterior distribution such as credible regions.

Developing from \cite{DiaconisFreedman1986}, we see (Section \ref{sec:weak-consistency}) that weak merging of Bayes and EB posterior distributions holds if and only if the EB posterior is weakly consistent, in the frequentist sense, at $\theta_0$, for every $\theta_0$ in the parameter space $\Theta$. Thus, conditions for weak consistency of EB posteriors are needed.
This provides a Bayesian motivation for studying frequentist consistency of EB posteriors.
Consistency is, however, of autonomous interest as well and we consider it in a general context, beyond the case of exchangeable (or i.i.d. $\sim P_{\theta_0})$ data. We provide sufficient conditions for
consistency of EB posteriors that cover both parametric and nonparametric cases.
In fact, an EB approach is even more tempting in nonparametric problems, since frequentist properties of Bayes procedures are known to crucially depend on the fine details of the prior \citep{DiaconisFreedman1986} and on a careful choice of the prior hyperparameters. We exhibit examples to illustrate EB consistency of Dirichlet process mixtures which is a commonly used nonparametric prior. Conditions for EB consistency seem stronger than those needed for consistency of Bayes posteriors: although for Bayes posteriors the so-called \emph{Kullback-Leibler prior support condition} implies weak consistency, it is not the case for EB posteriors. Thus, to validate EB posteriors we require stronger conditions and provide a counterexample where the EB posterior is not weakly consistent, despite validity of the Kullback-Leibler prior support condition.

Even when consistency and weak merging hold, simple examples show that the EB posterior can have unexpected and counterintuitive behaviors. Frequentist strong merging is a way to refine the analysis, allowing to perceive differences between Bayes and EB posterior distributions, even at first-order asymptotics.
Obtaining strong merging of Bayes posteriors in nonparametric contexts is often impossible since pairs of priors are typically singular.
Thus, in tackling this issue, we concentrate on parametric models and on the specific, but important, case of the \MMLE\, $\hatl_n$. In this setup, we find that the behavior of the EB posterior is essentially driven by the behavior of the prior at $\theta_0$.
Roughly speaking, if $\sup_{\lambda\in\Lambda}\pi(\theta_0|\lambda)= \infty$, then the EB posterior will tend to concentrate \vir{too much} around $\theta_0$ to merge with any posterior associated with a positive and continuous (at $\theta_0$) prior density. We illustrate this behavior in Bayesian regression with $g$-priors. Conversely, if $\sup_{\lambda\in\Lambda}\pi(\theta_0|\lambda)<\infty$, then frequentist strong merging holds and $\hat{\lambda}_n$ converges to the set of values for which such supremum is attained, say $\lambda_0$ assuming here uniqueness for ease of exposition. This value can be understood as the \textit{prior oracle} as it is the value of $\lambda$ that mostly favors $\theta_0$. In other terms, $\lambda_0$ is that value which allows to make the most precise statements \emph{a priori} on $\theta_0$. Under this respect, the EB posterior achieves some kind of optimality. An open issue is whether $\lambda_0$ is also {\em optimal} with respect to some
other criteria. We briefly discuss this and related open issues at the end of the paper.

We underline that, in the literature, the term \vir{empirical Bayes} is used with different meanings. One is that herein considered; yet, people also refer to empirical Bayes in problems where a \vir{prior} is introduced, but some frequentist interpretation of it is possible, typically in mixture models where $X_i|\lambda \overset{\textrm{i.i.d.}}{\sim}\int_{\Theta}p_\theta(\cdot)\,\mathrm{d}G(\theta|\lambda)$.
In this case, EB corresponds to maximum likelihood estimation of $\lambda$, \emph{i.e.}, MLE of the mixing distribution. A Bayesian approach to these problems would assign a prior distribution on $\lambda$. In this case, a comparison between Bayes and EB reduces to the interesting, but more standard, comparison between MLE and Bayes procedures and it is not studied in this note.

\subsection{General context and notation}\label{subsec:nota}
Let ${\cal X}$ and $\Theta$ denote the observational space and the parameter space, respectively. They can be very general, we only require that they are complete and separable metric spaces, equipped with their Borel $\sigma$-fields. Hereafter, any parameter space will follow the same assumption and, for such a space $\Theta$, the Borel $\sigma$-field is denoted by $\mathcal{B}(\Theta)$. Throughout the paper, unless otherwise stated, the same symbol $d$ is used to denote any metric (possibly semi-metric). Let $(X_i)_{i\geq 1}$ be a sequence of random elements, with the $X_i$'s taking values in $\cal X$. Suppose that the probability measure of $(X_i)_{i\geq 1}$ is modeled as a family $\{P_\theta^\infty:\,
\theta \in \Theta\}$. We denote by $\Ptn$ the joint probability law of $(X_1,\, \ldots,\, X_n)$.
We assume that $\Ptn$ is dominated by a common $\sigma$-finite measure $\mu$ (in principle, the dominating measure may depend on $n$, but we drop the $n$ to simplify the notation) and denote by $\ptn$ the density of $\Ptn$ w.r.t. $\mu$.

Let $\{\Pi(\cdot|\lambda):\,\lambda \in \Lambda\}$ be a family of prior probability distributions on $\Theta$. Given a prior $\Pi(\cdot|\lambda)$, we denote by $\Pi(\cdot|\lambda,\,X_1,\, \ldots,\, X_n)$ the corresponding posterior distribution of $\theta$, given $(X_1,\,\ldots,\,X_n)$. Since the model is dominated, the natural version of this conditional distribution is given by the Bayes' rule. In the sequel, we use the short notation $X_{1:n}:=(X_1,\,\ldots,\,X_n)$, $x_{1:n}:=(x_1,\,\ldots,\,x_n)$ and $x^\infty:=(x_1,\,x_2\ldots\,)$.

The EB approach consists in estimating the hyperparameter $\lambda$ by $\hatl_n\equiv\hatl_n(\Data)$ and plugging the estimate into the posterior distribution. In general, $\hatl_n$ takes values in the closure $\bar{\Lambda}$ of $\Lambda$. For $\lambda_0$ in the boundary $\partial\Lambda$ of $\Lambda$, we define $\Pi(\cdot|\lambda_0)$ as the weak limit of $\Pi(\cdot|\lambda)$ for $\lambda \rightarrow \lambda_0$, when it exists. We use the notation $P_n \Rightarrow P$ to mean that $P_n$ converges weakly to $P$, for any probabilities $P_n,\,P$.
%Note that the limit law is a finitely additive probability measure, not necessarily $\sigma$-additive (roughly speaking, the limit prior may be \vir{improper}).
We shall say that the EB posterior is well defined if $\Pi(\cdot|\hatl_n,\,\Data)$ is a probability measure for all large $n$, $P_\theta^\infty$-almost surely, for all $\theta$.
Then, the EB posterior is defined as
$$\Pi(B|\hatl_n,\,\Data)=\frac{\int_B \ptn(\Data)\,\mathrm{d}\Pi(\theta|\hatl_n)}
{\int_\Theta \ptn(\Data)\,\mathrm{d}\Pi(\theta|\hatl_n)} \qquad\mbox{for all Borel sets $B$.}
$$
In what follows, the EB posterior corresponding to the estimator $\hatl_n$
will be also denoted by $\Pi_{n,\,\hatl_n}$. Throughout the paper,
the EB posterior is always tacitly assumed to be well defined.

Many types of estimators $\hatl_n$ can be considered: in particular, the \MMLE, defined as
$$\hatl_n \in \argmax_{\lambda\in\Lambda}\, m(\Data|\lambda),\qquad
 \makebox{where }\,\,\, m(\Data|\lambda):=\int_\Theta \ptn(\Data)\,\mathrm{d}\Pi(\theta|\lambda),$$
is the most popular. Whenever we consider the \MMLE, we assume that $\sup_{\lambda\in\Lambda} m(\Data|\lambda)<\infty$ for all large $n$, $P_\theta^\infty$-almost surely, for all $\theta$,
and write $\hatm(\Data):= m(\Data|\hatl_n)$.
We shall present general results for the EB posterior without specifying the type of estimator $\hatl_n$ considered,
as well as specific results for the \MMLE.

The paper is organized as follows. In Section~\ref{sec:weak-consistency}, we study Bayesian merging and consistency of EB posteriors. Specifically, in subsection~\ref{sec:weakMerging}, we begin by studying weak merging of Bayes and EB procedures for exchangeable sequences. In subsection~\ref{sec:consistency},
we give sufficient conditions for consistency of the EB posterior for non-i.i.d. data.
Parametric and nonparametric examples are provided in subsection~\ref{sec:examples}.
In Section~\ref{sec:frequentistMerging}, we study frequentist strong merging and obtain,
as a by-product, that, when strong merging takes place, the EB procedure leads to an \emph{oracle} choice of the prior hyperparameter. Situations where strong merging fails are illustrated in Bayesian
regression with $g$-priors in subsection~\ref{sec:gpriors}.
Some open issues are discussed in Section~\ref{sec:final}.

%%%%%%%%%%%%%%%%%%%%%%%%%%%%%%%%%%%%%%%%%%%%%%%%%%%%%%%%%%%%%%%%%%%%%%%%%%%%%%%%%%%%%%%%%%%%%%%%%%%%%%%%%%%%%%%%%%%%

\section{Bayesian weak merging and consistency} \label{sec:weak-consistency}
We begin by studying {\em Bayesian weak merging} in the sense of \cite{DiaconisFreedman1986}
of Bayes and EB procedures in the basic case of exchangeable data.
Since Bayesian merging is intimately related to consistency, we then give conditions for consistency of the EB posterior in the general case.

\subsection{Bayesian merging of Bayes and empirical Bayes inferences} \label{sec:weakMerging}
Merging of Bayes and EB inferences formalizes the idea that
two posteriors or, in a predictive setting, two predictive distributions of all future observations, given the past, will eventually be close. A well-known result by \cite{blackwellDubins1962} establishes that, for $P$ and $Q$ probability laws of a process $(X_i)_{i\geq1}$, if $P$ and $Q$ are mutually absolutely continuous, then there is strong merging of the predictions of future events, given the increasing information provided by the data $\Data$. For exchangeable $P$ and $Q$ corresponding to priors $\Pi$ and $q$, respectively, $P$ and $Q$ are mutually absolutely continuous if and only if $\Pi$ and $q$ are such. However, the EB approach only gives a sequence of \vir{posterior distributions} $\Pi_{n,\,\hatl_n}$, without having a properly defined probability law of $(X_i)_{i\geq1}$.
Thus, the above result on strong merging does not apply. Furthermore, Blackwell and Dubins' result does not apply when the priors are singular, as it is often the case in nonparametric problems.

\cite{DiaconisFreedman1986} gave a notion of {\em weak merging} that applies even when strong merging does not. They showed that two Bayesians with different priors will merge weakly if and only if \emph{one} Bayesian has weakly consistent posterior, in the frequentist sense, at $\theta$, for every $\theta\in \Theta$. We show that an analogous result holds for Bayes and EB: in fact, EB merges with any true Bayes on $\Theta$ if and only if \emph{the} empirical Bayesian has posterior weakly consistent at $\theta$, for every $\theta\in \Theta$. Recall that two sequences of probability measures $p_n$ and $q_n$ are said to \emph{merge weakly} if and only if
$|\int g\,\mathrm{d}p_n - \int g\,\mathrm{d}q_n|\rightarrow0$ for all continuous and bounded functions $g$.

The results are herein restricted to the case of exchangeable sequences, thus, given $\theta$, the $X_i$'s are independent and identically distributed (i.i.d.) with common distribution $P_\theta$. We denote by $P_\theta^\infty$ the corresponding infinite product measure on $\cal{X}^\infty$.
 For any prior $\Pi$ on $\Theta$, the joint probability law of the parameter and the data is given by
$$
P_\Pi(A \times B):= \int_A P_\theta^\infty(B)\,\mathrm{d}\Pi(\theta)\qquad\mbox{for all Borel sets $A\subseteq\Theta$ and $B\subseteq\cal{X}^\infty$.}
$$
We denote by $\Pi_n$ the posterior distribution of $\theta$, given $\Data$. Note that here we do not need to assume a dominated model $\{\Ptn:\,\theta\in\Theta\}$, as long as there exists a natural version of the posterior distribution.
We also consider the \vir{predictive} distribution of all future observations $X_{n+1},\,X_{n+2},\,\ldots\,$, conditionally on $\Data$, given by
$$
P_{\Pi_n}(\cdot):=\int_\Theta P_\theta^\infty(\cdot)\,\mathrm{d}\Pi_n(\theta).
$$
% for all continuous functions $d$ satisfying  $d(\Pi,\,\Pi)=0$ . A choice of $d$ is a metric for the weak topology.
A posterior distribution $\Pi_n$ is {\em weakly consistent} at $\theta$ if, for any weak neighborhood $W$ of $\theta$, $\Pi_n(W^c) \rightarrow 0$ a.s.$\,[P_{\theta}^\infty]$, for all $\theta\in\Theta$.

Let us now consider the EB posterior $\Pi_{n,\,\hatl_n}$ described in subsection \ref{subsec:nota}.
The following result is a consequence of Theorem A.1 in \cite{DiaconisFreedman1986}.

\begin{prop} \label{prop:weakMerging}
Let the map $\theta \mapsto P_\theta$ be one-to-one and Borel. Given a family of priors
$\{\Pi(\cdot|\lambda):\,\lambda\in\Lambda\}$, let $\Pi_{n,\,\hatl_n}$ be the EB posterior.
The EB posterior is consistent at any $\theta\in\Theta$
if and only if, for any prior probability $q$ on $\Theta$,
the EB posterior and the Bayes posterior $q_n$ weakly merge, as $n \rightarrow \infty$, with $P_q$-probability $1$.

Moreover, if the map $\theta \mapsto  P_\theta$ is continuous and has continuous inverse,
then the EB posterior is consistent at any $\theta$ if and only if, for all priors $q$ on $\Theta$,
the predictive distributions $P_{\Pi_{n,\,\hatl_n}}$ and $P_{q_n}$ weakly merge, with $P_q$-probability $1$, that is,
$$
\forall\,f\in C_B(\mathcal{X}^\infty),\quad\abs{\int_{\mathcal{X}^\infty} f\,\mathrm{d}(P_{\Pi_{n,\, \hatl_n}}-P_{q_n})}\rightarrow0\qquad \mathrm{a.s.}\,[P_q].
$$
\end{prop}

\begin{proof}
It suffices to note that the proof for the equivalences $(i)$--$(iv)$ in Theorem~A.1
of \cite{DiaconisFreedman1986}, page~18, goes through to the present case: in fact, it is based on the properties of the Bayes posterior $q_n$, which hold true in this case, whereas, for the EB posterior $\Pi_{n,\,\hatl_n}$, only consistency is required.
\end{proof}

Proposition \ref{prop:weakMerging} shows that any Bayesian is sure that her/his estimate w.r.t. the quadratic loss of any continuous and bounded function $g$ will asymptotically agree with the EB estimate if and only if the EB posterior is consistent at any $\theta$. If so, in particular, a Bayesian with prior $\Pi(\cdot|\lambda)$ is sure that
$|\int g(\theta)\,\mathrm{d}\Pi(\theta|\hatl_n,\,\Data) - \int g(\theta)\,\mathrm{d}\Pi(\theta|\lambda,\,\Data)| \rightarrow 0$.

\subsection{Consistency of empirical Bayes posteriors} \label{sec:consistency}
Weak merging gives a Bayesian motivation for consistency as a \emph{conditio sine qua non} for inter-subjective agreement, as seen in Proposition \ref{prop:weakMerging}. Of course, consistency is a basic property in itself from a frequentist viewpoint. Therefore, it is of interest to study consistency of the EB posterior distribution in greater generality, beyond the case of i.i.d. (or exchangeable) observations.
Consistency of parametric and nonparametric Bayesian procedures is fairly well understood. However, sufficient conditions for consistency of a Bayesian posterior are not enough for consistency of the EB posterior, since, in the latter case, the EB prior is data-dependent through $\hatl_n$.
We give two results on consistency of EB posteriors that hold true for dependent sequences and cover both parametric and nonparametric cases.
To be more specific, let $(\Theta,\,d)$ be a semi-metric space. For any $\epsilon>0$, let $U_\epsilon\equiv U_\epsilon(\theta_0):=\{\theta\in\Theta:\,d(\theta,\theta_0)<\epsilon\}$ denote the
open ball centered at $\theta_0$ with radius $\epsilon$. Note that $d(\cdot,\,\cdot)$ is understood as a loss function and, depending on the model, it can be a distance directly on $\theta$, for instance $(\theta- \theta')^2$ if $\theta$ is real, or a (pseudo-)distance on the model $p_\theta^{(n)}$, for instance the Hellinger distance between densities in a density model or the $L_2$-norm between regression functions in a regression model.
We provide sufficient conditions so that, for any $\epsilon>0$, the posterior probability \[\Pi(U_\epsilon^c|\hatl_n,\,\Data)\rightarrow0 \qquad\textrm{a.s.}\,[P_0^\infty],\]
where $P_0^\infty$ denotes the probability measure of $(X_i)_{i\geq1}$, under $\theta_0$.

\subsubsection{Sufficient conditions for consistency of empirical Bayes posteriors}
Recall that the model $\Ptn$ is dominated by a $\sigma$-finite measure $\mu$ and has density
$\ptn$ w.r.t. $\mu$. For $\theta\in\Theta$, let \[R(\ptn):=\frac{\ptn}{\pto}(\Data)\]
denote the likelihood ratio. We shall use the following assumptions.
\begin{itemize}
\item[$(\textbf{A1})$]There exist constants $c_1,\,c_2>0$ such that, for any $\epsilon>0$,
  \[P_0^*\pt{\sup_{\theta\in U_\epsilon^c}R(\ptn)
  \geq e^{-c_1n\epsilon^2}}\leq c_2 (n\epsilon^2)^{-(1+t)}\]
  for some $t>0$, where $P_0^*$ denotes the outer measure.
\item[$(\textbf{A2})$]For each $\theta_0\in\Theta$, there exists
      $\lambda_0\in\Lambda$ such that, for any $\eta>0$,
      \[\Pi(B_{\mathrm{KL}}(\theta_0;\,\eta)|\lambda_0)>0,\]
      where, for $\mathrm{KL}_\infty(\theta_0;\,\theta):= - \lim_{n\rightarrow\infty}{n}^{-1}\log
      R(\ptn)$, the set $B_{\mathrm{KL}}(\theta_0;\,\eta):=\{\theta\in\Theta:\,
      \mathrm{KL}_\infty(\theta_0;\,\theta)<\eta\}$.
\end{itemize}

\medskip

\noindent
When compared to the assumptions usually considered for posterior consistency,
$(\textbf{A1})$ is quite strong; it is, however, a common assumption in the maximum likelihood estimation literature. It is verified in most parametric models, see, \emph{e.g.}, \cite{schervish95}, and also in nonparametric models.
For instance, \cite{wongShen1995} proved that, for \emph{i.i.d.} observations with density $f_\theta$,
if $U_\epsilon$ is the Hellinger open ball centered at $f_0:=f_{\theta_0}$ with radius $\epsilon$, that is, $U_\epsilon\equiv H_\epsilon(f_0):=\{f_\theta:\,h(f_\theta,\,f_0)<\epsilon\}$,
a \emph{sufficient} condition for $(\textbf{A1})$ to hold true is that there exist
constants $c_3,\,c_4>0$ such that, for each $\epsilon>0$,
\begin{equation}\label{cond:ws:entropy}
\int_{\epsilon^2/2^8}^{\sqrt{2}\epsilon}\sqrt{H_{[\,]}(u/c_3,\,\Theta,\,h)}\,\mathrm{d}u\leq c_4\sqrt{n}\epsilon^2\qquad\mbox{ for $n$ large enough,}
\end{equation}
where the function $H_{[\,]}(\cdot,\,\Theta,\,h)$ denotes the Hellinger
bracketing metric entropy of $\Theta$.

In this paper, $(\textbf{A1})$ is used to
handle the numerator of the ratio defining the EB posterior probability of any neighborhood $U_\epsilon$ in the following way. By the first Borel-Cantelli lemma, $(\textbf{A1})$ implies that
$\sup_{\theta\in U_\epsilon^c}R(\ptn)<e^{-c_1n\epsilon^2}$ for all large $n$, a.s.~$[P_0^\infty]$.
Thus,
\begin{equation}\label{rem:WS}
\hspace*{-0.5cm}\int_{U_\epsilon^c}R(\ptn)\,\mathrm{d}\Pi(\theta|\hatl_n)\leq \Pi(U_\epsilon^c|\hatl_n)
\sup_{\theta\in U_\epsilon^c}R(\ptn)\leq \sup_{\theta\in U_\epsilon^c}R(\ptn)<e^{-c_1n\epsilon^2}\quad\textrm{for all large }n,
\end{equation}
$P_0^\infty$-almost surely. Note that the bound in \eqref{rem:WS} is valid for \emph{any} type of estimator
$\hatl_n$.

Assumption $(\textbf{A2})$ is the usual \emph{Kullback-Leibler prior support condition}, herein required to hold true for some value $\lambda_0\in\Lambda$. It is a mild assumption considered in most results on posterior
consistency and has been shown to be satisfied for various models and families of priors. Note that the rather abstract definition of $\mathrm{KL}_\infty(\cdot;\,\cdot)$ is mainly considered to deal with the non-i.i.d case. In the i.i.d. case, $\mathrm{KL}_\infty(\theta_0;\,\theta)$ is simply the Kullback-Leibler divergence between the densities $p_{\theta_0}$ and $p_{\theta}$ (per observation).
In the present context, it is used when $\hatl_n$ is the \MMLE\,
to bound from below $\hatm(\Data)/\pto(\Data)$. For other types of estimator $\hatl_n$,
a variant of $(\textbf{A2})$ is considered, cf. conditions $(ii)$-$(iii)$ of Proposition~\ref{thm1}.

\medskip

As explained in subsection~\ref{subsec:nota}, there are infinitely many possible choices for
$\hatl_n$. We first study consistency of the EB posterior when
$\hatl_n$ is the \MMLE, which is a common EB approach, see, for instance, \cite{berger:85}, \cite{GeorgeandFoster2000}, \cite{scott:berger:10}, just to name but a few. We then consider the case
where $\hatl_n$ is any estimator for which some direct knowledge is available, another common EB approach, even without any explicit mention of it being an EB procedure, see \cite{McABJ2006} for an example of plug-in EB in a nonparametric setting.

%%%%%%%%%%%%%%%%%%%%%%%%%%%%%%%%%%%%%%%%%%%%%%%%%%%%%%%%%%%%%%%%%%%%%%%%%%%%%%%%%%%%%%%%%%%%%%%%%%%%

\subsubsection{Case of the maximum marginal likelihood estimator}
Recall that $\hatl_n\in\argmax_{\lambda\in\Lambda}\, m(\Data|\lambda)$ and $\hatm(\Data)<\infty$.
We have the following result.

\begin{prop}\label{prop1}
Under assumptions $(\mathbf{A1})$ and $(\mathbf{A2})$,
for any $\epsilon>0$,
\[\Pi(U_\epsilon^c|\hat{\lambda}_n,\,\Data)\rightarrow0\qquad\mathrm{a.s.}\,[P_0^\infty].\]
\end{prop}
\begin{proof}
Fix $\epsilon>0$. Set $N_n:=\int_{U_\epsilon^c}R(\ptn)
\,\mathrm{d}\Pi(\theta|\hatl_n)$, under $(\mathbf{A1})$, by \eqref{rem:WS},
$N_n<e^{-c_1n\epsilon^2}$ for all large $n$,
$P_0^\infty$-almost surely. Let $D_n:=\int_{\Theta} R(\ptn)
\,\mathrm{d}\Pi(\theta|\hatl_n)$. By definition of $\hatm(\Data)$, with $P_0^\infty$-probability 1,
\[D_n=\frac{\hatm(\Data)}{\pto(\Data)}
\geq\frac{m(\Data|\,\lambda_0)}{\pto(\Data)}=:D_n(\lambda_0)\qquad\mbox{for all large }n,\]
where $\lambda_0$ is as required in $(\mathbf{A2})$.
Reasoning as in Lemma~10 of \cite{barron1988},
page~23, for any $\eta>0$, $D_n(\lambda_0)>e^{-n\eta}$ for all large $n$, $P_0^\infty$-almost surely.
Choosing $0<\eta<c_1\epsilon^2$, for $\delta:=(c_1\epsilon^2-\eta)>0$, we have
$\Pi(U_\epsilon^c|\hatl_n,\,\Data)=N_n/D_n\leq N_n/D_n(\lambda_0)<e^{-n\delta}$ for all large $n$,
$P_0^\infty$-almost surely. The assertion follows.
\end{proof}

\begin{rmk}
\emph{Although it seems intuitive that the usual Kullback-Leibler condition on the
positivity of the prior mass of Kullback-Leibler neighborhoods of $\Pto$, say $(\mathbf{A2})$, implies \emph{weak} consistency of the EB posterior, as it happens for the true posterior, it is, however, not the case and additional assumptions on the behavior of the likelihood ratio and/or on the prior need to be required, as illustrated in the following example. Consider \cite{Bahadur1958}'s example, see also \cite{LC98}, pages~445--447, and \cite{ghoshRamamoorthi2003}, pages 29--31. Let $\Theta=\N^*$. For each $\theta=k$, a density $p_\theta$ on $[0,\,1]$
is defined as follows. Let $a_0=1$ and define recursively $a_k$ by
$\int_{a_k}^{a_{k-1}}[h(x)-C]\,\mathrm{d}x =1-C$, where $0<C<1$ is a given constant
and $h(x)=e^{1/x^2}$. Since $\int_0^1e^{1/x^2}\,\mathrm{d}x=\infty$, the
$a_k$'s are uniquely determined and the sequence $a_k\rightarrow0$ as $k\rightarrow\infty$.
For $\theta\in\Theta$, define
\begin{eqnarray*}
p_\theta(x) = \left\{ \begin{array}{cl}
h(x), & \mbox{if} \quad a_\theta < x \leq a_{\theta-1},\\[5pt]
C, & \mbox{if} \quad x\in [0,1] \cap ( a_\theta,\, a_{\theta-1}]^c,\\[5pt]
0, & \mbox{otherwise.}
\end{array}
\right.
\end{eqnarray*}
Let $X_1,\,\ldots,\,X_n|\theta\overset{\textrm{i.i.d.}}{\sim}p_\theta$.
The MLE $\hat{\theta}_n$ exists and tends to $\infty$ in probability,
regardless of the true value $\theta_0=k_0$ of $\theta$. It is, therefore, inconsistent.
On the other hand, $\Theta$ being countable, by Doob's theorem,
any proper prior on $\Theta$ leads to a consistent posterior at \emph{all} $\theta\in\Theta$.
Consider a family of priors $\{\Pi(\cdot|\lambda):\,\lambda\in\Lambda\}$ such that, for each $\theta$, there exists $\lambda \in \bar{\Lambda}$ for which $\Pi(\cdot| \lambda) = \delta_{\theta}$. If $\lambda\in\partial\Lambda$, then $\Pi(\cdot| \lambda)$ is defined as the weak limit of any sequence $\Pi(\cdot|\lambda_j)$ for $\lambda_j \rightarrow \lambda$ (when it exists). It is always possible to construct such a family of priors.
For $\lambda = (m,\, \sigma)$, let
\begin{eqnarray*}
\Pi(1|\lambda) &:=& \Phi( (1/2 - m)/\sigma) - \Phi((-1/2 - m)/\sigma),\\
\Pi(\theta|\lambda) &:=& \Phi( (\theta-1/2 - m)/\sigma) - \Phi((\theta -3/2 - m)/\sigma)\\&&\hspace*{3cm} + \,\Phi( (-\theta-3/2 - m)/\sigma)-\Phi( (-\theta-1/2- m)/\sigma)\qquad\mbox{ for $\theta>1$},
\end{eqnarray*}
where $\Phi$ is the cumulative distribution function of a standard Gaussian random variable.
By taking $m=\theta-1$ and letting $\sigma\rightarrow0$, we have as a limit the Dirac mass at $\theta$ because
$\Pi(1|\lambda)=0$ and $\Pi(\theta|\lambda)\rightarrow1$.
Thus, for each $k_0 \in \N^*$, by taking $m=k_0-1$ and letting $\sigma\rightarrow0$, we have as a limit the Dirac mass at $k_0$. Then, the EB posterior is the Dirac mass at the MLE $\hat{\theta}_n$,
which is inconsistent. To see this, note that
$$\forall\,\lambda \in \Lambda,\quad m(\Data|\lambda) \leq \prod_{i=1}^np_{\hat{\theta}_n}(X_i) \qquad \mbox{ and } \qquad  \hatm(\Data)=m(\Data|(\hat{\theta}_n-1,\, 0))=\prod_{i=1}^np_{\hat{\theta}_n}(X_i).$$}
\end{rmk}

\begin{rmk}\label{rmk:rates}
\emph{Proposition~\ref{prop1} gives a result on consistency, however, replacing $\epsilon$ with $\bar{\epsilon}_n$ in $(\mathbf{A1})$ and with $\tilde{\epsilon}_n$ in the stronger version of $(\mathbf{A2})$ as found in
\cite{GvdV071}, $\epsilon_n:=(\bar{\epsilon}_n\vee \tilde{\epsilon}_n)$ is an upper bound on the contraction rate
for the EB posterior. On the other hand, $\epsilon_n$ is an upper bound also on the rate of convergence for the Bayes posterior $\Pi(\cdot|\lambda_0,\,\Data)$. Choosing the value $\lambda_0$ that leads to the best prior concentration rate $\tilde{\epsilon}_n$ results in the best posterior rate $\epsilon_n$, at least if $\bar{\epsilon}_n\leq \tilde{\epsilon}_n$, in which case, the EB procedure attains some kind of optimality. We will precise this kind of results in the parametric framework in Section 3.}
\end{rmk}

%%%%%%%%%%%%%%%%%%%%%%%%%%%%%%%%%%%%%%%%%%%%%%%%%%%%%%%%%%%%%%%%%%%%%%%%%%%%%%%%%%%%%%%%%%%
\subsubsection{Case of a convergent $\hatl_n$}\label{subsec:conv-lambda}
In some applications, $\hatl_n$  is chosen to be a convenient statistic, like
some moment estimator, so that the prior is centered at a plausible
value for the parameter. In such cases, $\hatl_n$ has often a known
asymptotic behavior,  which does not necessarily mean that the EB posterior should have a stable behaviour, even if the prior has. In the following proposition, we give sufficient conditions
for consistency of the EB posterior in such situations. Suppose that the parameter is split into
$\theta=(\tau,\,\zeta)$, where $\tau\in\mathrm{T}$ and $\zeta\in\mathrm{Z}$ and, given $\lambda \in \Lambda \subseteq \R^\ell$, $\tau \sim \tilde{\Pi}(\cdot|\lambda)$ while $\zeta \sim \tilde{\Pi}$. In other words, the hyperparameter $\lambda$ only influences the prior distribution of $\tau$.
The overall prior is $\Pi(\cdot|\lambda):=\tilde{\Pi}(\cdot|\lambda)\times\tilde{\Pi}(\cdot)$.
Let $\theta_0=(\tau_0,\,\zeta_0)$ be the true value of $\theta=(\tau,\,\zeta)$.

\begin{prop}\label{thm1}
Let $\tilde{\Pi}(\cdot|\hatl_n)$, $n=1,\,2,\,\ldots$,
and $\tilde{\Pi}(\cdot|\lambda_0)$ be probability measures on $\mathcal{B}(\mathrm{T})$.
Assume that $(\mathbf{A1})$ is satisfied and
\begin{itemize}
\item[$(i)$] $\tilde{\Pi}(\cdot|\hatl_n)
\Rightarrow\tilde{\Pi}(\cdot|\lambda_0)$ a.s.$\,[P_0^\infty]$,
\item[$(ii)$] for each $\eta>0$, there exists a set $K_\eta\subseteq B_\mathrm{KL}(\theta_0;\,\eta)$ such that $\Pi(K_\eta|\lambda_0)>0$,
\item[$(iii)$] defined, for each $x^{\infty}\in\mathcal{X}^\infty$ and any $\eta>0$, the set
\[E_{x^\infty}^{(\eta)}:=
\pg{(\tau,\,\zeta)\in K_\eta:\,
\frac{1}{n}\log\frac{\ptzo}{\ptnz}(\data)\nrightarrow\mathrm{KL}_\infty((\tau_0,\,\zeta_0);\,(\tau,\,\zeta))\quad\textrm{for some }\tau_n\rightarrow\tau},\]
$E_{x^\infty}^{(\eta)}\in\mathcal{B}(\mathrm{T})\otimes\mathcal{B}(\mathrm{Z})$ and, for $P_0^\infty$-almost every $x^\infty\in\mathcal{X}^\infty$,
\begin{equation}\label{eq:nullmeasu}
\Pi(E_{x^\infty}^{(\eta)}|\lambda_0)=0.
\end{equation}
\end{itemize}
Then, for any $\epsilon>0$,
\[\Pi(U_\epsilon^c|\hatl_n,\,\Data)\rightarrow0\qquad\mathrm{a.s.}\,[P_0^\infty].\]
\end{prop}
%%%%%%%%%%%%%%%%%%%%%%%%%%%%%%%%%%%%%%%%%%%%%%%%%%%%%%%%%%%%%%%%%%%%%%%%%%%%%%%%%%%%%

\medskip

The proof of Proposition \ref{thm1} is postponed to the Appendix.

\medskip

\begin{rmk}\label{GC}
\emph{Condition $(i)$ is a natural condition in those cases where $\hatl_n$ is an explicit estimator (as opposed to the \MMLE) such as a moment type estimator, see Example 2, Gaussian DP mixture\,--\,II, in subsection~\ref{sec:examples}. Condition $(ii)$ is the usual Kullback-Leibler prior support condition, except for the fact that here it concerns the support of the limiting prior. Condition $(iii)$ is more unusual. If, in the definition of $E^{(\eta)}_{x^\infty}$, the $\tau_n$'s were fixed at $\tau$, then $(iii)$ would be a basic ergodic condition on the support of $\Pi(\cdot|\lambda_0)$, so the difficulty comes from obtaining an ergodic theorem uniformly over neighborhoods of $\tau$. In the case of \emph{i.i.d.} observations, the following condition implies $(iii)$. If
\[\quad\forall\,\eta,\,\epsilon>0,\,\,\,\forall\,\theta\in K_\eta,\,\,\,\exists\,
\delta\equiv\delta(\theta,\,\epsilon)>0:
\quad\mathbb{E}_0\pq{\sup_{\theta'\in \Theta:\,d(\theta',\,\theta)<\delta}
\abs{\log\dfrac{p_\theta}{p_{\theta'}}(X_1)}}<\dfrac{\epsilon}{2},
\]
then standard SLLN arguments imply that there exists a set $\mathcal{X}^\infty_0\subseteq\mathcal{X}^\infty$, with $P_0^\infty(\mathcal{X}^\infty_0)=1$, such that, for each $x^\infty\in\mathcal{X}^\infty_0$, condition
\eqref{eq:nullmeasu} is satisfied.}
\end{rmk}

%%%%%%%%%%%%%%%%%%%%%%%%%%%%%%%%%%%%%%%%%%%%%%%%%%%%%%%%%%%%%%%%%%%%%%%%%%%%%%%%%%%%%%%%%%%%%%%%%%%%%%%%%%

\bigskip

We now present some examples illustrating the above consistency results.

\subsection{Examples}\label{sec:examples}
We begin by considering a  parametric example. Even if very simple, this example is illuminating since it illustrates the different types of behaviour of consistent EB posteriors to be expected when $\hatl_n$ is the \MMLE. These phenomena are studied in greater generality in Section~\ref{sec:frequentistMerging}, where it is shown that
the behavior of the EB posterior and the \MMLE\, is driven by the behavior of the map $\lambda\mapsto \pi(\theta_0|\lambda)$, where $\pi(\theta_0|\lambda)$ is the prior density, given $\lambda$, evaluated at
$\theta_0$.\\

\noindent
{\bf Example 1: Parametric case}\\[7pt]
Let $X_1,\,\ldots,\,X_n|\theta\sim \ptn$, $\theta\in\mathbb{R}$, with
$\{\ptn:\,\theta\in\mathbb{R}\}$ such that $(\mathbf{A1})$ holds true.
Let $\theta$ be given a Gaussian prior distribution, with mean $m$ and variance $\tau^2$,
$\theta\sim\textrm{N}(m,\,\tau^2)$. Consider the EB posterior with the \MMLE\, $\hatl_n$
in the following three cases:
\begin{itemize}
\item[$(1)$] $\tau^2$ is fixed and $\lambda=m$ is estimated,\\[-0.9cm]
\item[$(2)$] $m$ is fixed and $\lambda=\tau^2$ is estimated,\\[-0.9cm]
\item[$(3)$] $\lambda=(m,\,\tau^2)$ and both parameters are estimated.
\end{itemize}
Interestingly, the behavior of the EB posterior $\Pi(\cdot|\hatl_n,\,\Data)$ is quite different
in the three cases. As an illustration, we consider the simple case where $X_i|\theta \overset{\textrm{i.i.d.}}{\sim}\textrm{N}(\theta,\,\sigma^2)$, with $\sigma^2$ known,
which satisfies $(\mathbf{A1})$. Indeed, the choice of the sampling model is of little
consequence to the asymptotic behaviour of $\hatl_n$.

\begin{description}
\item[{\bf Case (1)}.]
The posterior distribution of $\theta$ corresponding to a fixed value $\tau^2$
is $\textrm{N}(m_n,\,(1/\tau^2 + n/\sigma^2)^{-1})$, where
\[m_n:=\frac{\sigma^2/n}{\sigma^2/n+\tau^2} \,\lambda+
\frac{\tau^2}{\tau^2 + \sigma^2/n} \,\bar{X}_n=\frac{\sigma^2/n}{\sigma^2/n+\tau^2} \, m+
\frac{\tau^2}{\tau^2 + \sigma^2/n} \, \bar{X}_n.\]
The EB posterior is obtained plugging the \MMLE, $\hatl_n=\bar{X}_n$, and it is $\textrm{N}(\bar{X}_n,\,(1/\tau^2 + n/\sigma^2)^{-1})$, which has a \emph{completely regular} density.
Both posteriors can be seen to be consistent by direct computations.

It is worth making a comparison with the \emph{hierarchical} Bayes, which typically assigns a normal hyperprior on $\lambda=m$, that is, $\lambda \sim \textrm{N}(\lambda_0,\,\tau^2_0)$ so that
$\Pi^h(\theta)=\textrm{N}(\theta|\lambda_0,\,\tau^2+\tau^2_0)$. Note that the hierarchical prior increases the uncertainty on $\theta$. The posterior is
$\textrm{N}(m_n^h,\,(1/(\tau^2+\tau_0^2)+n/\sigma^2)^{-1})$, where $m_n^h$ has the same expression as $m_n$, with $\tau^2$ replaced by $\tau^2+\tau_0^2$.
\end{description}
\begin{description}
\item[{\bf Case (2)}.] Recall that from \cite{LC98}, page 263, when $\lambda=\tau^2$, with $m$ fixed,
\[ \sigma^2+n\hat{\tau}^2=\max\{\sigma^2,\,n(\bar{X}_n-m)^2\},\,\,\,\,\,\,\,\,\mbox{so that}\,\,\,\,\,\,\,\, \hat{\tau}^2=\frac{\sigma^2}{n}\max\pg{\frac{n(\bar{X}_n-m)^2}{\sigma^2}-1,\,0}.\]
The EB posterior $\Pi(\cdot|\hat{\tau}^2,\,\Data)$ is normal with mean and variance having the same expressions as in Case (1), with $\tau^2$ replaced by $\hat{\tau}^2$.
A \emph{hierarchical} Bayes approach would assign a prior on $\tau^2$,
\emph{e.g.}, $1/\tau^{2} \sim \textrm{Gamma}(\nu/2,\, 2/\nu)$.
This leads to a Student's-$t$ prior distribution for $\theta$, with \vir{flatter} tails, that
may give better frequentist properties, see, for example, \cite{BergerRobert1990}, \cite{BergerStrawderman1996}. However, the Student's-$t$ prior is no longer conjugate and the EB posterior is simpler to compute.

In this example, the EB posterior is only \emph{partially regular}, in the sense that $\hat{\tau}^2$ can be equal to zero with positive probability so that
$\Pi(\cdot|\hat{\tau}^2,\,\Data)$ is degenerate at $m$, although, in the case where $m\neq\theta_0$, this probability converges to zero. This type of behavior is discussed more extensively in subsection~\ref{sec:gpriors}, where it is shown that, if $m\neq \theta_0$, the EB posterior merges strongly with any regular posterior, including the hierarchical posterior, with probability going to $1$, whereas, if $m=\theta_0$, there is a positive probability that the EB posterior does not merge strongly with any regular posterior.

\end{description}
\begin{description}
\item[{\bf Case (3)}.] In this case, the \MMLE\, for $\lambda=(m_0,\, \tau^2)$ is $\hatl_n=(\bar{X}_n,\, 0)$. The EB posterior is \emph{completely irregular} in the sense that it is always degenerate at $\bar{X}_n$. Note that this example is much more general than the Gaussian case and applies, in particular, to any location-scale family of priors. Indeed, if the model $\ptn$ admits a MLE $\hat{\theta}_n$ and $\pi(\cdot|\lambda)$ is of the form $\sigma^{-1}g((\cdot-\mu)/\sigma)$, with $\lambda=(\mu,\,\sigma)$, for some unimodal density $g$ which is maximum at $0$, then $\hatl_n=(\hat{\theta}_n,\,0)$ and the EB posterior is the point mass at $\hat{\theta}_n$. This shows that such families of priors should not be used in combination with \MMLE\, EB procedures.
\end{description}

\smallskip

Next, two nonparametric examples concerning Dirichlet process (DP)
location-scale and location mixtures of Gaussians are exhibited:
in the first one, the \MMLE\, for the precision parameter of the DP base measure
is considered, in the second one, a moment type estimator for the mean of
a normal base measure is employed.\\

\noindent{\bf Example 2: Nonparametric case}
\begin{description}
\item[{\bf Gaussian DP mixture\,--\,I}.]
Consider the following nonparametric model of Gaussian location-scale mixtures: the observations $X_i|G\overset{\textrm{i.i.d.}}{\sim} p_G(\cdot):=\int\phi(\cdot|\mu,\,\sigma^2)\,\mathrm{d}G(\mu,\,\sigma)$, where $\phi(\cdot|\mu,\,\sigma^2)$ denotes the density of a Gaussian random variable with mean $\mu$ and variance $\sigma^2$. In this example, $\theta=G $ belongs to the set $\Theta$ of all probability measures on $\R\times\R^{+*}$. We assume that $G \sim \mathrm{DP}(\alpha(\cdot|\lambda))$, where $\alpha(\cdot|\lambda)$ denotes a family of positive and finite measures on $\R\times\R^{+*}$. \cite{Liu1996} and \cite{McABJ2006} consider EB procedures in this type of models. In particular, \cite{Liu1996} considers the \MMLE\, for $\lambda=\alpha(\R\times\R^{+*})$, fixing the base probability measure $\bar{\alpha}(\cdot):=\alpha(\cdot)/\alpha(\R\times\R^{+*})$. Even if he considers a mixture of Binomial distributions, the argument remains valid for other types of Dirichlet process mixtures. For the sake of simplicity, we present our computations in the case of Gaussian mixtures. Following \cite{Liu1996}, see also \cite{PetroneRaftery97}, $\hatl_n$ is the solution of
\begin{equation}\label{eq:liu}
\sum_{j=1}^n \frac{\lambda}{\lambda+j-1}=\mathbb{E}[K_n|\lambda,\,\Data],
\end{equation}
where $\mathbb{E}[K_n|\lambda,\,\Data]$ is the expected number of occupied clusters under the conditional posterior distribution, given $\lambda$. If we assume that $\bar{\alpha}$ has support $A\times[\underline{\sigma},\,\bar{\sigma}]$, with $A$ a compact interval of $\R$, $0<\underline{\sigma}\leq\bar{\sigma}<\infty$ and $\Theta=\{G:\,\mathrm{supp}(G)\subseteq A\times[\underline{\sigma},\,\bar{\sigma}]\}$, then, from Theorem~3.2 of \cite{GvdV01}, page~1244, $\{p_G:\,G\in\Theta\}$ has bracketing Hellinger metric entropy satisfying condition \eqref{cond:ws:entropy}, so that assumption $(\mathbf{A1})$ is fulfilled. Moreover, if $p_{G_0}$ is a mixture of Gaussian distributions, $p_{G_0}(\cdot):=\int \phi(\cdot|\mu,\,\sigma^2)\,\mathrm{d}G_0(\mu,\,\sigma)$,
with $\mathrm{supp}(G_0)\subseteq A\times[\underline{\sigma},\,\bar{\sigma}]$, then also condition $(\mathbf{A2})$ is satisfied. The existence of a solution of \eqref{eq:liu} implies that the EB posterior is well defined, thus, using Proposition \ref{prop1}, we get consistency for the EB posterior.

\smallskip

%%%%%%%%%%%%%%%%%%%%%%%%%%%%%%%%%%%%%%%%%%%%%%%%%%%%%%%%%%%%%%%%%%%%%%%%%%%%%%%%%%%%%%%%%%%%%%%%%%%%%%%%%%%%%
\item{{\bf Gaussian DP mixture\,--\,II}.}
Consider the following model of Gaussian location mixtures:
$X_1,\,\ldots,\,X_n|(F,\,\sigma)\overset{\mathrm{i.i.d.}}{\sim}
p_{F,\,\sigma}(\cdot):=\int\phi(\cdot|\mu,\,\sigma^2)\,\mathrm{d}F(\mu)$.
In this case, $\theta=(F,\,\sigma)$ is assumed to take values in $\Theta:=\{(F,\,\sigma):\,F(\mathbb{R})=1,\,\,\sigma\in[\underline{\sigma},\,\bar{\sigma}]\}$, with
$0<\underline{\sigma}\leq\bar{\sigma}<\infty$.
Letting $\alpha(\cdot):=\alpha(\mathbb{R})\bar{\alpha}(\cdot)$,
with fixed precision $0<\alpha(\mathbb{R})<\infty$ and probability measure
$\bar{\alpha}$ specified, up to the mean, as $\textrm{N}(\mu_{\bar{\alpha}},\,\sigma^2_{\bar{\alpha}})$,
we assume that $F\sim\mathrm{DP}(\alpha)$ and $\sigma\sim H$, with $\mathrm{supp}(H)=[\underline{\sigma},\,\bar{\sigma}]$. In this case, $\lambda=\mu_{\bar{\alpha}}\in\R$, for which the estimator $\hatl_n:=\bar{X}_n$ is considered. Let $\bar{\alpha}_n(\cdot) := \mathrm{N}(\bar{X}_n,\,\sigma^2_{\bar{\alpha}})$ and the un-normalized EB base measure $\hat{\alpha}_n(\cdot):=\alpha(\mathbb{R})\bar{\alpha}_n(\cdot)=
\alpha(\mathbb{R})\mathrm{N}(\bar{X}_n,\,\sigma^2_{\bar{\alpha}})$ so that the EB prior on $(F,\,\sigma)$ is  $\mathrm{DP}(\hat{\alpha}_n)\times H$.
We prove consistency of the EB posterior w.r.t. the Hellinger distance $h$ or the $L_1$-distance.
Let $m_0:=\mathbb{E}_0[X_1]$ be the mean of $X_1$ under $p_{F_0,\,\sigma_0}$, with $\sigma_0\in[\underline{\sigma},\,\bar{\sigma}]$ and $F_0$ satisfying $F_0([-a,a]^c)\lesssim e^{-c_0 a^2 }$ for all large $a$ and a constant $c_0>0$. For fixed $\epsilon>0$, choose $0<\delta<\epsilon^2$ small enough and $a_n= n^{q}$, with $1/2< q < 1$. Consider the sieve set $\Theta_n:=\{(F,\,\sigma):\,F([-a_n,\,a_n])>1-\delta,\,\,\,\sigma\in[\underline{\sigma},\,\bar{\sigma}]\}$.
From Theorem 6 in \cite{GvdV072}, combined with the proof of Theorem 7 in \cite{GvdV072}, if $\tilde{\Theta}_n:=  \{(F,\,\sigma):\,F([-a_n,\,a_n])=1,\,\,\,\sigma\in[\underline{\sigma},\,\bar{\sigma}]\}$, for all $\epsilon^2 / 2^8 < u < \sqrt{2}\epsilon$,  $H_{[\,]}(u/c_3,\,\Theta_n,\,h)\lesssim a_n(\log a_n  + \log(1/\epsilon))^2$. Thus, for $n$ large enough,
$\int_{\epsilon/2^8}^{\sqrt{2}\epsilon}(H_{[\,]}(u/c_3,\,\Theta_n,\,h))^{1/2}\,\mathrm{d}u\lesssim\epsilon \sqrt{a_n}(\log a_n) <\epsilon^2\sqrt{n}$ because $a_n = n^q$ with $q < 1$.
By \eqref{rem:WS}, $P_0^\infty$-almost surely, $\int_{H_\epsilon^c\cap\Theta_n} R(p^{(n)}_{F,\,\sigma})\,\mathrm{d}\Pi(F,\,\sigma|\hatl_n)<e^{-c_1n\epsilon^2}$ for all large $n$.
We now show that the expected value of the integral over $\Theta_n^c$ tends to $0$.
In this case, $\lambda_0=m_0$. Define $\alpha_0(\cdot):=\alpha(\mathbb{R})\bar{\alpha}_0(\cdot)=\alpha(\mathbb{R})\mathrm{N}(m_0,\,\sigma^2_{\bar{\alpha}})$. Since $\bar{X}_n\overset{\mathrm{a.s.}}{\longrightarrow} m_0$, we have $\bar{\alpha}_n\Rightarrow\bar{\alpha}_0$, $\mathrm{a.s.}\,\,[P_0^\infty]$, whence, by Theorem~3.2.6 in \cite{ghoshRamamoorthi2003}, pages~105--106, $\mathrm{DP}(\hat{\alpha}_n)\Rightarrow\mathrm{DP}(\alpha_0)$, $\mathrm{a.s.}\,\,[P_0^\infty]$, and
condition $(i)$ of Proposition~\ref{thm1} is fulfilled. Denote by $\Pi(\cdot|\lambda_0)$ the overall limiting prior $\mathrm{DP}(\alpha_0)\times H$. Since $F\sim\mathrm{DP}(\hat{\alpha}_n)$, using the stick-breaking representation, we have $p_{F,\,\sigma}(\cdot)=\sum_{j=1}^\infty p_j\phi_\sigma(\cdot-\xi_j)$, with $\xi_j\sim \textrm{N}(\bar{X}_n,\,\sigma^2_{\bar{\alpha}})$. As $\phi_\sigma(\cdot-\xi_j)=
\phi_{\sigma}(\cdot-(\bar{X}_n-m_0)-\xi_j')$, with
$\xi'_j\sim \textrm{N}(m_0,\,\sigma^2_{\bar{\alpha}})$, we have $p_{F,\,\sigma}(\cdot)=
p_{F',\,\sigma}(\cdot-(\bar{X}_n-m_0))$, with $F'\sim \mathrm{DP}(\alpha_0)$.
Let $A_n$ be the set wherein the inequality $|\bar{X}_n-m_0|\leq L/\sqrt{n}$ holds true for some constant $L>0$. Note that $P_{\theta_0}^{(n)}(A_n^c)$ can be made as small as needed by choosing $L$ large enough. Using the above representation of the EB Dirichlet prior,
\[
\begin{split}
p_{F',\,\sigma}^{(n)}(\Data - (\bar{X}_n - m_0))  &\leq \prod_{i=1}^n \int_{\R} \phi_\sigma(X_i - \xi) e^{ \frac{ L|X_i - \xi| }{\sqrt{n} \sigma^2}}\,\mathrm{d}F'(\xi)\\
 &= c_{n,\, \sigma}^n \prod_{i=1}^n \int_\R g_{\sigma}(X_i  - \xi)\,\mathrm{d}F'(\xi),
\end{split}
\]
where $g_\sigma$ is the probability density proportional to $ \phi_\sigma(y) e^{L|y|/(\sqrt{n}\sigma^2)}$ and
$$c_{n,\,\sigma}:= \int_\R \phi_\sigma(y) e^{L|y|/(\sqrt{n}\sigma^2)}\,\mathrm{d}y \leq e^{L^2 /(2n\sigma^2)}\left( 1 + \frac{2L}{ \sigma\sqrt{n}}\right),$$
which implies that
$$  \mathbb{E}_0\pq{\mathbf{I}_{A_n}(\Data)\int_{\Theta_n^c} R(p^{(n)}_{F,\,\sigma})\,\mathrm{d}\Pi(F,\,\sigma|\hatl_n)}\lesssim \left( 1 +\frac{ 2L}{\underline{\sigma}\sqrt{n}} \right)^n \Pi(\Theta_n^c|\lambda_0)\lesssim e^{ -c_1 a_n^2 +  c_2 \sqrt{n}} \leq e^{-c_3a_n^2},  $$
for $n$ large enough, by definition of $a_n$.

Next, we bound from below the denominator of the ratio defining the EB posterior probability of the set $H_\epsilon^c$. Using similar computations to those above, on $A_n$,
\begin{equation*}
 \int_\Theta R(p^{(n)}_{F,\,\sigma})\,\mathrm{d}\Pi(F,\,\sigma|\hatl_n) \gtrsim e^{- \frac{ L^2 }{ 2\underline{\sigma}^2 } }c_{n,\,\sigma}^{n} \int_\Theta R(\tilde{g}^{(n)}_{F',\,\sigma})\,\mathrm{d}\Pi(F',\,\sigma|\lambda_0),
 \end{equation*}
with $\tilde{g}_{F',\,\sigma}(\cdot) = c_{n,\,\sigma}^{-1}\int_{\R} \phi_\sigma(\cdot-\xi)e^{-L\frac{|\cdot-\xi|}{\sqrt{n}\sigma^2}} \,\mathrm{d}F'(\xi)$, where, with abuse of notation, we still denote by $c_{n,\,\sigma}$ the normalizing constant, although it is not exactly the same as before. Using similar computations to those in the proof of (5.21) in \cite{GvdV01}, we obtain that, for any $\eta >0$,
 $\int_\Theta R(\tilde{g}^{(n)}_{F',\,\sigma})\,\mathrm{d}\Pi(F',\,\sigma|\lambda_0)\geq\exp{\{-n\eta\}} $, for $n$ large enough. Consistency of the EB posterior follows.
\end{description}

%%%%%%%%%%%%%%%%%%%%%%%%%%%%%%%%%%%%%%%%%%%%%%%%%%%%%%%%%%%%%%%%%%%%%%%%%%%%%%%%%%%%%%%%%%%%%%%%%%%%%%%%

We have seen in Example 1 that, even in simple models, the EB posterior under the \MMLE\, may be degenerate, which, from a Bayesian perspective, is a pathological behaviour. In the following section, we explain why and when such behaviour is to be expected. The key factor is the choice of the family of priors $\{\Pi(\cdot|\lambda):\,\lambda\in \Lambda\}$. This has strong practical implications since it shows that, unless some kind of strong shrinkage is required, it is better avoid maximizing over scale hyperparameters.

\section{Frequentist strong merging and asymptotic behavior of $\hatl_n$}\label{sec:frequentistMerging}

For each $\lambda\in\Lambda\subseteq\mathbb{R}^{\ell}$, $\ell\in\mathbb{N}$,
let $\Pi(\cdot|\lambda)$ be a prior on $\Theta\subseteq\mathbb{R}^k$, $k\in\mathbb{N}$, with density $\pi(\cdot|\lambda)$ w.r.t. a common measure $\nu$. Before formally stating a general result which describes the asymptotic behaviour of EB posteriors, we present an informal argument to explain the heuristics behind it. Under usual regularity conditions on the model, the marginal distribution, given $\lambda$, can be thus approximated
$$ m(\Data| \lambda)  = \pi(\theta_0|\lambda) \frac{ p_{\hat{\theta}_n}^{(n)}(\Data)(2\pi)^{k/2} }{ n^{k/2}|I(\theta_0)|^{1/2}}(1 + o_p(1))$$
under $\Pto$. If we could interchange the maximization and the limit, we would have
 $$\argmax_{\lambda\in \Lambda}\, m(\Data|\lambda) = \argmax_{\lambda\in \Lambda}\, \pi(\theta_0|\lambda)+ o_p(1).$$
An interesting phenomenon occurs: assuming the above argument is correct, the \MMLE\, is asymptotically maximizing the family $\{\pi(\theta_0|\lambda):\, \lambda \in \Lambda\}$, where $\theta_0$ is the true value of the parameter. In other words, it selects the most interesting value of $\lambda$ in terms of the prior model. We call the set
of values of $\lambda$ maximizing $\pi(\theta_0|\lambda)$ the \textit{prior oracle set of hyperparameters} and denote it by $\Lambda_0$. In terms of (strong) merging, however, $\Lambda_0$ may correspond to unpleasant values, typically if
$$\sup_{\lambda \in \Lambda}\pi(\theta_0|\lambda)= \infty $$
and $\Pi(\cdot|\lambda_0)$ is the Dirac mass at $\theta_0$. Then, the EB posterior is degenerate.
This is what happens in Cases (2) and (3) of Example 1 or, more generally, when $\pi(\cdot|\lambda)$ is a location-scale family and $\lambda $ contains the scale parameter.
Obviously, in such cases, we cannot interchange the limit and the maximization. We now present these ideas more rigorously.

Let $d$ be a semi-metric on $\Theta$ and, for any $\epsilon>0$, let $U_\epsilon$ denote the
open ball centered at $\theta_0$ with radius $\epsilon$.
The map $g:\,\theta\mapsto\sup_{\lambda\in\Lambda}\,\pi(\theta|\lambda)$
from $\Theta$ to $\mathbb{R}^+$ induces a partition $\{\Theta_0,\,\Theta_0^c\}$ of $\Theta$, with
$\Theta_0:=\{\theta\in\Theta:\,g(\theta)<\infty\}$ and $\Theta_0^c:=\{\theta\in\Theta:\,g(\theta)=\infty\}$.
We refer to the case where $g(\theta_0)<\infty$ or, equivalently, $\theta_0\in\Theta_0$,
as the \emph{non-degenerate} case and to the complementary case as the \emph{degenerate} case. As illustrated in the above heuristic discussion and in Sections \ref{subsec:nondege} and \ref{subsec:dege} below, the use of this terminology is motivated by the fact that, in the former case, the EB posterior is regular, whereas, in the latter case, it tends to be \vir{too much concentrated} at $\theta_0$ to merge strongly with any regular Bayes posterior.

%%%%%%%%%%%%%%%%%%%%%%%%%%%%%%%%%%%%%%%%%%%%%%%%%%%%%%%%%%%%%%%%%%%%%%%%%%%%%%%%%%%%%%%%%%%%%%%%%%%%%%%%%%%

\subsection{Non-degenerate case} \label{subsec:nondege}
We give sufficient conditions for the EB posterior $\Pi(\cdot|\hatl_n,\,\Data)$,
where $\hatl_n$ is the \MMLE, to \emph{merge strongly} with \emph{any} posterior $\Pi(\cdot|\lambda,\,\Data)$, $\lambda\in\Lambda$, in the non-degenerate case. The possibility of having a degeneracy
as in Case (2) of Example 1 is ruled out by assuming $\theta_0\in\Theta_0$.
When $\theta_0\in\Theta_0$, we define the set $\Lambda_0:=\{\lambda_0\in\Lambda:\,\pi(\theta_0|\lambda_0)=g(\theta_0)\}$ and
the subset $\tilde{\Lambda}_0 \subseteq\Lambda_0$ of all those $\lambda_0\in\Lambda_0$ for which
the map $\theta\mapsto \pi(\theta|\lambda_0)$ is continuous at $\theta_0$ and, for any $\epsilon,\,\eta>0$, $\Pi(U_\epsilon\cap B_\mathrm{KL}(\theta_0;\,\eta)|\lambda_0)>0$.

\begin{thm}\label{th:strongmerg}
Suppose that $\theta_0\in \Theta_0$. Assume that $(\mathbf{A1})$ is satisfied and
\begin{itemize}
\item[$(i)$]the map $g:\,\theta\mapsto\sup_{\lambda\in\Lambda}\,\pi(\theta|\lambda)$
is positive and continuous at $\theta_0$,
\item[$(ii)$]$\tilde{\Lambda}_0 \neq\emptyset$,
\end{itemize}
then, for each $\lambda_0\in\tilde{\Lambda}_0$,
\begin{equation}\label{convergence}
\frac{\hatm(\Data)}{m(\Data|\lambda_0)}\rightarrow1 \qquad \mathrm{a.s.}\,[P_0^\infty].
\end{equation}

\medskip

\noindent If, in addition to $(i)$ and $(ii)$, the following assumption is satisfied
\begin{itemize}
\item[$(iii)$]$\tilde{\Lambda}_0=\Lambda_0$ is included in the interior of $\Lambda$ and, for any
$\delta>0$, there exist $\epsilon,\,\eta>0$ so that
\[\sup_{\theta\in U_{\epsilon}}\, \sup_{\lambda\in\Lambda:\,d(\lambda,\,\Lambda_0)>\delta}\frac{\pi(\theta|\lambda)}{g(\theta)}\leq1-\eta,\]
where $d(\lambda,\,\Lambda_0):=
\inf_{\lambda_0\in\Lambda_0}d(\lambda,\,\lambda_0)$,
\end{itemize}
then
\begin{equation}\label{eq:merging}
d(\hat{\lambda}_n,\,\Lambda_0)\rightarrow0\qquad \mbox{and} \qquad \|\pi(\cdot|\hatl_n,\,\Data)-\pi(\cdot|\lambda_0,\,\Data)\|_1\rightarrow0\qquad\mathrm{a.s.}\,[P_0^\infty].
\end{equation}
\end{thm}

\medskip

The proof of  Theorem \ref{th:strongmerg} is presented in the Appendix.
Some remarks and comments are in order here.

\begin{rmk}
\emph{A key set of values of $\lambda$ for understanding the asymptotic behavior of the EB posterior is therefore the \emph{prior oracle set of hyperparameters}, $\Lambda_0$, which consists of all those values for which the prior mostly favors $\theta_0$. Equation~\eqref{eq:merging} shows that, $P_0^\infty$-almost surely, $\hatl_n$ \vir{converges} to $\Lambda_0$. This result implies that the EB procedure asymptotically selects the \emph{smartest} values of $\lambda$ for \vir{estimating} $\theta_0$. Furthermore, in this case the EB posterior does not degenerate.}
\end{rmk}

\begin{rmk}
\emph{Under the conditions of Theorem~\ref{th:strongmerg}, if $\pi$ is any prior density w.r.t. $\nu$ which is positive and continuous at $\theta_0$ and whose posterior is consistent at $\theta_0$, then the EB posterior merges strongly with the posterior corresponding to the prior $\Pi$. This is a direct consequence of Theorem~\ref{th:strongmerg} combined with Theorem 1.3.1 of \cite{ghoshRamamoorthi2003}, pages~18--20. In particular, the EB posterior merges strongly with any \emph{hierarchical} Bayes posterior associated with a prior $\Pi^h$, for $h$ a prior on $\Lambda$, provided the map $\theta\mapsto\pi^{h}(\theta):=\int_{\Lambda} \pi(\theta|\lambda)\,\mathrm{d}h(\lambda)$ is positive and continuous at $\theta_0$ and the posterior corresponding to $\Pi^{h}$ is consistent at $\theta_0$ in terms of the priors $\Pi(\cdot|\lambda)$, $\lambda \in \Lambda$.}
\end{rmk}

\subsection{Degenerate case and extension to the model choice framework} \label{subsec:dege}

Theorem~\ref{th:strongmerg} implies that, if $g(\theta_0)<\infty $, under smoothness assumptions on $\pi(\theta|\lambda)$ for $\theta $ in a neighborhood of $\theta_0$, the EB posterior will eventually be close to any Bayes posterior based on a smooth prior density. Example 1, Cases (2) and (3), suggests that this might not be the case when $g(\theta_0)=\infty $ and there exists a $\lambda_0 \in \bar{\Lambda}$ for which $\Pi(\cdot|\lambda_0)=\delta_{\theta_0}$. In Section \ref{app:pr:dege} of the Appendix, we show that, for such $\theta_0$, the $L_1$-distance between the EB posterior and the Bayes posterior associated with any $\lambda\in \Lambda$ is bounded from below on a set whose probability remains asymptotically strictly positive, so that no strong merging can take place. This \emph{critical} phenomenon cannot be improved by having greater smoothness in the likelihood. Indeed, consider usual regularity assumptions on the model, \emph{i.e.},  $(\mathbf{A1})$ is satisfied, for any $\epsilon>0$ and $\theta \in U_\epsilon$,
  $$ l_n(\theta)-l_n(\hat{\theta}_n) \in - \frac{ n(\theta - \hat{\theta}_n )' I(\theta_0) (\theta - \hat{\theta}_n) }{ 2 }  ( 1 \pm \epsilon),\qquad \hat{\theta}_n \mbox{ denoting the MLE,}$$
and $l_n(\hat{\theta}_n)-l_n(\theta_0)$ converges in distribution to a $\chi^2$-distribution with $k$ degrees of freedom. Also, assume that $\ptn$ is bounded as a function of $\theta$ for all $n$.
Then, if $\theta_0 \in \Theta_0^c$ and there exists $\lambda_0 \in \bar{\Lambda}$ such that $\Pi(\cdot |\lambda_0) = \delta_{\theta_0}$, the EB posterior cannot merge strongly with any posterior $\Pi(\cdot|\lambda,\,\Data)$, with $\lambda \in \Lambda$ such that the prior density $\pi(\cdot|\lambda)$ is positive and continuous at $\theta_0$. In particular, it cannot merge with any hierarchical posterior.

Interestingly, \cite{scott:berger:10} also encounter this phenomenon in their comparison between fully Bayes and EB approaches for variable selection in regression models. We believe this is due to the same reasons as described above, although it does not completely fit the setup we have described because we have restricted ourselves to priors that are absolutely continuous w.r.t. Lebesgue measure. However, this is not a crucial difference. We describe in an informal way the link between our explanation above and their findings. First, we briefly recall their setup. They consider a regression model
$$Y_i=\alpha+\bs{X}_i'\boldsymbol{\beta}+\epsilon_i, \qquad \epsilon_i \sim \textrm{N}(0,\,\phi^{-1}), \quad\bs{X}_i \in \R^m, \quad m >1,$$
and their aim is to select the best set of covariates among the $m$ candidates.
They consider the following hierarchy:
given a model indexed by the inclusion vector $\boldsymbol{\gamma}\in\{ 0,\, 1\}^m$,
where $\gamma_i=1$ means that the $i$th covariate belongs to the model $M_{\boldsymbol{\gamma}}$, consider a prior
$\pi_{\boldsymbol{\gamma}}(\boldsymbol{\theta})=\pi(\alpha,\,\phi)\pi_{k_{\boldsymbol{\gamma}}}(\boldsymbol{\beta})$, where $\pi_{k_{\boldsymbol{\gamma}}}(\boldsymbol{\beta})$ is absolutely continuous w.r.t. Lebesgue measure on $\R^{k_{\boldsymbol{\gamma}}}$ as a distribution on the coefficients included in the model $\boldsymbol{\gamma}$.
Then, they consider $\Pi(M_{\boldsymbol{\gamma}}|p)=p^{k_{\boldsymbol{\gamma}}}(1-p)^{m-k_{\boldsymbol{\gamma}}}$, for $p\in (0,\,1)$, and study the EB approach which consists in computing the \MMLE\, for $p$. The marginal likelihood is written as
$$ m( \bs{Y}|p)=\sum_{\boldsymbol{\gamma}} p^{k_{\boldsymbol{\gamma}}}(1-p)^{m-k_{\boldsymbol{\gamma}}} m_{{\boldsymbol{\gamma}}}(\bs{Y}). $$
Each model is regular so that, under $P_{\boldsymbol{\theta}_0}^{(n)}$, with ${\boldsymbol{\theta}}_0=(\alpha_0,\,\boldsymbol{\beta}_0,\,\phi_0)$,
$$\frac{m_{\boldsymbol{\gamma}}(\bs{Y})}{p^{(n)}_{\boldsymbol{\theta}_0}}\approx
\frac{c_{\boldsymbol{\gamma}}\pi(\alpha_0,\,\phi_0) \pi_{k_{\boldsymbol{\gamma}}}(\boldsymbol{\beta}_0)e^{l_n(\hat{\boldsymbol{\theta}}_{\boldsymbol{\gamma}}) - l_n(\boldsymbol{\theta}_0)}}{ n^{(k_{\boldsymbol{\gamma}}+2)/2} },  \qquad c_{\boldsymbol{\gamma}}= (2\pi)^{(k_{\boldsymbol{\gamma}}+2)/2} |I_{\boldsymbol{\gamma}}(\theta_0)|^{-1/2},$$
for all ${\boldsymbol{\gamma}}$ such that we cannot have $\gamma_j=0$ and $\beta_{0j}\neq 0$, otherwise, the marginal is exponentially small. Hence, if ${\boldsymbol{\beta}}_0 =\bs{0}$, the above marginals are maximized (in ${\boldsymbol{\gamma}}$) at ${\boldsymbol{\gamma}}=\bs{0}$ due to the Ockham's-razor effect of integration and, with probability going to $1$, the EB posterior distribution puts mass $1$ on $M_{\bs{0}}$. Generally speaking, if $\boldsymbol{\theta}_0$ belongs to a model with $k_0$ covariates, the EB posterior will concentrate on $\hat{p}=k_0/m$ with probability going to $1$. Again, this is an oracle value since, in terms of prior on the models, it is centered at the right number of covariates, however, if $\boldsymbol{\theta}_0$ is either in the null or in the largest model, it corresponds to a completely degenerate prior on the models.

This is not merely specific of the linear regression example, and, in a general model choice framework with competing models $M_j$, $j=1,\, \ldots, J$, when $\lambda$ corresponds to a hyperparameter on the distribution of models, it is not only the prior values $\pi(\theta_0|\lambda)$ which drive the behaviour of the EB, but rather the values
$$\sum_{j:\,\theta_0\in M_j}\frac{\pi_j(\theta_0|M_j) P(M_j|\lambda)}{n^{d_j/2}}$$
due to integration over the parameters in $M_j$. Thus, $\hatl_n$ asymptotically maximizes
$P(M^*|\lambda)$, where $M^*$ is the smallest model containing $\theta_0$.
Depending on the form of $P(M^*|\lambda)$, the EB \textit{prior} distribution can be degenerate or not.

In the following section, we describe more carefully regression with $g$-priors, which has been considered in the literature in combination with EB procedures.
%%%%%%%%%%%%%%%%%%%%%%%%%%%%%%%%%%%%%%%%%%%%%%%%%%%%%%%%%%%%%%%%%%%%%%%%%%%%%%%%%%%%%%%%%%%%%%%%%%%%%%%%%%%%%%

\subsection{Example: Regression with $g$-priors}\label{sec:gpriors}
Consider the canonical Gaussian regression model $\bs{Y}=\bs{1}\alpha+\bs{X}\boldsymbol{\beta}+\boldsymbol{\varepsilon}$, with $\boldsymbol{\varepsilon}\sim\textrm{N}_n(\bs{0},\,\sigma^2\bs{I})$,
where $\bs{Y}=(Y_1,\,\ldots,\,Y_n)'$ is the response vector,
$\bs{1}$ denotes the vector of $1$'s of length $n$, $\alpha$ is the intercept,
$\bs{X}$ is the $n\times k$ \emph{fixed} design matrix, $\boldsymbol{\beta}=(\beta_1,\,\ldots,\,\beta_k)'$
is the $k$-dimensional vector of regression coefficients, $\sigma^2$ is the error variance and
$\bs{I}$ denotes the $n\times n$ identity matrix. Clearly, $\bs{Y}|\alpha,\,\boldsymbol{\beta},\,\sigma^2\sim \textrm{N}_n(\bs{1}\alpha+\bs{X}\boldsymbol{\beta},\,\sigma^2\bs{I})$.
Note that $Y_i|\alpha,\,\boldsymbol{\beta},\,\sigma^2\overset{\textrm{ind.}}{\sim}
\textrm{N}(\alpha+\bs{x}_i'\boldsymbol{\beta},\,\sigma^2)$, $i=1,\,\ldots,\,n$,
where $\bs{x}_i'$ is the $i$th row of $\bs{X}$: the $Y_i$'s are (conditionally) independent,
but \emph{not} identically distributed. Let $\tilde{\bs{X}}$ denote the design matrix whose columns have been re-centered so that $\bs{1}'\tilde{\bs{X}}=\bs{0}'$, in which case $\boldsymbol{\beta}$ can be estimated separately from $\alpha$ using OLS estimators. The following condition is assumed throughout:
for fixed $1\leq k<n$, the matrix $n^{-1}(\tilde{\bs{X}}'\tilde{\bs{X}})$ is positive definite and converges to a positive definite matrix $\bs{V}$.
We consider the following default prior specification for $\alpha,\,\boldsymbol{\beta},\,\sigma^2$:
\begin{equation}\label{eq:g-prior}
\pi(\alpha,\,\sigma^2)\propto\frac{1}{\sigma^2},\qquad\qquad
\boldsymbol{\beta}|\sigma^2\sim\textrm{N}_k(\bs{0},\,g
\sigma^2(\tilde{\bs{X}}'\tilde{\bs{X}})^{-1}),\quad \mbox{with }g>0,
\end{equation}
where the prior covariance matrix of $\boldsymbol{\beta}$ is a scalar multiple $g$ of the covariance matrix
$\sigma^2(\tilde{\bs{X}}'\tilde{\bs{X}})^{-1}$ of the OLS estimator $\hat{\boldsymbol{\beta}}$ of $\boldsymbol{\beta}$. The prior mean for $\boldsymbol{\beta}$ can, in principle, be any $\boldsymbol{\beta}_a\in\mathbb{R}^k$, nonetheless, we take $\boldsymbol{\beta}_a=\bs{0}$
because this choice helps keeping the presentation at a simple technical level, without any loss of generality for the purpose of this study. The prior in \eqref{eq:g-prior}, which is a modified version of the original \cite{Zellner1986}'s $g$-prior, is widely used in the variable selection literature, see, \emph{e.g.}, \cite{ClydeandGeorge2000}, \cite{GeorgeandFoster2000}, \cite{Liangetal2008}.
In this example, $\lambda=g$, therefore we consider the EB posterior of $\boldsymbol{\beta}$ with the \MMLE\, of $g$, which, from equation (9) in
\cite{Liangetal2008}, page~413, is known to be
\[\hat{g}_n:=\max\{F_n-1,\,0\},\qquad\quad\mbox{ where }\,F_n:=\frac{R^2/k}{(1-R^2)/(n-1-k)},\]
$R^2$ being the ordinary coefficient of determination. Suppose that $\bs{Y}$
is generated by the model with parameter values $\alpha_0,\,\boldsymbol{\beta}_0,\,\sigma^2_0$.
Let $(\Omega,\mathcal{F},\,\mathbb{P})$ denote the probability space wherein $\bs{Y}$ is defined.
It turns out that
\begin{equation}\label{eq:g-conv}
\left\{\begin{array}{ll}\varliminf_{n\rightarrow\infty}\mathbb{P}(\hat{g}_n=0)
=\varliminf_{n\rightarrow\infty}\mathbb{P}(F_n\leq1)\geq\gamma>0,  &\hbox{\textrm{ if } $\boldsymbol{\beta}_0=\bs{0}$,} \\[7pt]
\lim_{n\rightarrow\infty}\mathbb{P}(\hat{g}_n>0)=\lim_{n\rightarrow\infty}\mathbb{P}(F_n>1)=1, &\hbox{\textrm{ if } $\boldsymbol{\beta}_0\neq\bs{0}$.}
\end{array}
\right.
\end{equation}
Interestingly, when $\boldsymbol{\beta}_0=\bs{0}$, even if the prior guess on the value of $\boldsymbol{\beta}_0$ is correct, the probability that the \MMLE\, takes a value in the boundary (thus causing the EB posterior to be degenerate) does not asymptotically vanish. Conversely, when $\boldsymbol{\beta}_0\neq\bs{0}$, even if the prior guess on $\boldsymbol{\beta}_0$ is wrong,
the probability that the EB posterior is non-degenerate tends to $1$, as $n\rightarrow\infty$.
To prove \eqref{eq:g-conv}, let
\[\tilde{F}_n:=\frac{(\hat{\boldsymbol{\beta}}-\boldsymbol{\beta}_0)'
(\tilde{\bs{X}}'\tilde{\bs{X}})(\hat{\boldsymbol{\beta}}-\boldsymbol{\beta}_0)/k}{\textrm{SSE}/(n-1-k)}.\]
\begin{itemize}

\item[$\bullet$] If $\boldsymbol{\beta}_0=\bs{0}$, then $F_n\equiv\tilde{F}_n\overset{\mathrm{a.s.}}{\longrightarrow}\chi^2_k/k$, because $\mathrm{SSE}/(n-k-1)\overset{\mathrm{a.s.}}{\longrightarrow}\sigma_0^2$, and $\varliminf_{n\rightarrow\infty}\mathbb{P}(\hat{g}_n=0)\geq\mathbb{P}(\chi^2_k/k\leq1)=:\gamma>0$;

\item[$\bullet$] if $\boldsymbol{\beta}_0\neq\bs{0}$, from consistency of $\hat{\boldsymbol{\beta}}$, \emph{i.e.}, $\hat{\boldsymbol{\beta}}\overset{\mathrm{a.s.}}{\longrightarrow}\boldsymbol{\beta}_0$,
\[R_n:=\frac{n^{-1}[(\boldsymbol{\beta}_0-2\hat{\boldsymbol{\beta}})'(\tilde{\bs{X}}'
\tilde{\bs{X}})\boldsymbol{\beta}_0]/k
}{\textrm{SSE}/(n-1-k)}\overset{\mathrm{a.s.}}{\longrightarrow}-\,\frac{(
\boldsymbol{\beta}_0'\bs{V}\boldsymbol{\beta}_0)/k
}{\sigma_0^2}<0,\]
which implies that $1+nR_n\rightarrow-\infty$. Consequently,
\[\mathbb{P}(\hat{g}_n>0)=\mathbb{P}(F_n>1)=
\mathbb{P}\pt{\tilde{F}_n>1+\frac{[(\boldsymbol{\beta}_0-2\hat{\boldsymbol{\beta}})'(\tilde{\bs{X}}'
\tilde{\bs{X}})\boldsymbol{\beta}_0]/k
}{\textrm{SSE}/(n-1-k)}}=\mathbb{P}(\tilde{F}_n>1+nR_n)\rightarrow1.\]
\end{itemize}
We now study the consequences of \eqref{eq:g-conv} on \emph{frequentist merging in total variation}.
Some preliminary remarks are in order. For each $g>0$, the posterior
$\Pi(\cdot|g,\,\bs{Y})$ of $\boldsymbol{\beta}$ is absolutely continuous
w.r.t. Lebesgue measure on $\mathbb{R}^k$. Let $\pi(\cdot|g,\,\bs{Y})$ denote its density.
By direct computations, whatever $\boldsymbol{\beta}_0\in\mathbb{R}^k$, for each
$g>0$, $\Pi(\cdot|g,\,\bs{Y})\Rightarrow\delta_{\boldsymbol{\beta}_0}$, $\mbox{a.s.}\,[\mathbb{P}]$.
Defined the set $\Omega_n:=\{\hat{g}_n=0\}$, clearly, $\Omega_n\subseteq\{\Pi(\cdot|\hat{g}_n,\,
\bs{Y})=\delta_{\bs{0}}\}$. We have the following results.
\begin{itemize}

\item[$\bullet$] If $\boldsymbol{\beta}_0=\bs{0}$ then, for each $g>0$,
$\varliminf_{n\rightarrow\infty}\mathbb{P}
    (d_{\mathrm{TV}}(\Pi(\cdot|g,\,\bs{Y}),\,\Pi(\cdot|\hat{g}_n,\,
    \bs{Y}))=1)>0$, where $d_{\mathrm{TV}}$ denotes the total variation distance.
Therefore, even if the prior guess on $\boldsymbol{\beta}_0$ is correct,
there is a set of positive probability wherein \emph{strong merging cannot take place}.
In fact, on $\Omega_n$, for the Borel set $A=\{\bs{0}\}$, we have $\Pi(A|g,\,\bs{Y})=0$ because $A$ has null Lebesgue measure. Then, $1\geq d_{\mathrm{TV}}(\Pi(\cdot|g,\,\bs{Y}),\,\Pi(\cdot|\hat{g}_n,\,
\bs{Y}))\geq|\Pi(A|g,\,\bs{Y})-\delta_{\bs{0}}(A)|=1$.
Thus, $\varliminf_{n\rightarrow\infty}\mathbb{P}(\{d_{\mathrm{TV}}(\Pi(\cdot|g,\,\bs{Y}),\,\Pi(\cdot|\hat{g}_n,\,
\bs{Y}))=1\})\geq\varliminf_{n\rightarrow\infty}\mathbb{P}(\Omega_n)\geq\gamma>0$.

\item[$\bullet$] If $\boldsymbol{\beta}_0\neq\bs{0}$ then, for each $g>0$, $\mathbb{P}(\|\pi(\cdot|g,\,\bs{Y})-\pi(\cdot|\hat{g}_n,\,\bs{Y})\|_1\rightarrow0)\rightarrow1$,
where $\pi(\cdot|\hat{g}_n,\,\bs{Y})$ denotes Lebesgue density of the EB posterior. The result assures that, even if the prior guess on $\boldsymbol{\beta}_0$ is wrong, \emph{strong merging takes place} on a set with probability tending to $1$. This is a direct consequence of Theorem \ref{th:strongmerg}.
\end{itemize}

%%%%%%%%%%%%%%%%%%%%%%%%%%%%%%%%%%%%%%%%%%%%%%%%%%%%%%%%%%%%%%%%%%%%%%%%%%%%%%%%%%%%%%%%%%%%%%%%%%%%%%%%%%%%

\section{Final remarks}\label{sec:final}

In this paper, we discussed whether the common knowledge that an
EB posterior is asymptotically equivalent to other Bayes procedures is correct.
We first discussed weak merging as a minimal requirement to motivate, from a Bayesian viewpoint, the use of EB procedures. Along the lines of \cite{DiaconisFreedman1986}'s results on consistency of Bayesian procedures,
we showed that merging is intimately related to consistency of the EB posterior and some general conditions on weak consistency are provided. Weak merging is, however, too weak a criterion to shed light on some pathological examples related to EB posteriors. Hence, we also studied strong merging in a more restricted framework, which has enabled us to characterize the families of priors which could lead to those \textit{degenerate} behaviours.

In Section~\ref{sec:frequentistMerging}, we showed that, at least for finite-dimensional parameter spaces, the EB procedure with the \MMLE\, asymptotically selects the {\em oracle} value of the hyperparameter, that is the value for which the prior mostly favors $\theta_0$. An open issue is whether this value also leads to optimal frequentist asymptotic properties of the EB posterior. In nonparametric problems, for instance, frequentist asymptotic properties of Bayes procedures may crucially depend on the fine details of the prior: in particular, the posterior may have sub-optimal or optimal contraction rate depending on the value of a hyperparameter $\lambda$ and there exists a value $\lambda^*$ entailing the minimax-optimal rate. An open question is whether the \emph{oracle} value $\lambda_0$ may be optimal in this sense.

A different problem related to the one herein investigated is that of maximum likelihood estimation and Bayesian inference for exchangeable data. The data would be {\em physically} exchangeable with probability law $P_\lambda$, rather than {\em physically} independent (i.i.d. according to $P_\theta$). Then, the EB approach would account for maximum likelihood estimation of $\lambda$, whereas Bayesian inference would put a prior on $\lambda$. In this case, there would be a {\em true} value $\lambda_0$ of $\lambda$ and frequentist asymptotic properties should be studied w.r.t. $P_{\lambda_0}^\infty$, rather than w.r.t. $P_0^\infty\equiv P_{\theta_0}^\infty$.

%%%%%%%%%%%%%%%%%%%%%%%%%%%%%%%%%%%%%%%%%%%%%%%%%%%%%%%%%%%%%%%%%%%%%%%%%%%%%%%%%%%%%%%%%%%%%%%%%%%%%%%%%%%

\bigskip
\noindent{\bf Acknowledgements}.
This work originated from a question by Persi Diaconis. We are grateful to him and Jim Berger for stimulating discussions.
S.P. has been partially supported by the Italian Ministry of University and Research, grant 2008MK3AFZ.

\bigskip

\section{Appendix}

\subsection{Proof of Proposition~\ref{thm1}}\label{app:pr:th1}

Fix $\epsilon>0$. Set $N_n:=\int_{U_\epsilon^c}R(\ptn)
\,\mathrm{d}\Pi(\theta|\hatl_n)$, under $(\mathbf{A1})$,
$N_n<e^{-c_1n\epsilon^2}$ for all large $n$,
$P_0^\infty$-almost surely. Let $D_n:=\int_{\Theta} R(\ptn)
\,\mathrm{d}\Pi(\theta|\hatl_n)$. In order to bound from below
$D_n$, it is convenient to refer to the probability space,
say $(\Omega,\,\mathcal{F},\,\mathbb{P})$, wherein the $X_i$'s are defined.
Let
\[
\mu_n^{(\omega)}(\cdot):=\tilde{\Pi}(\cdot|\hatl_n(\omega)),\,\,\,n=1,\,2,\,\ldots,\qquad\mbox{ and }\qquad
\mu_0(\cdot):=\tilde{\Pi}(\cdot|\lambda_0).
\]
Let $\Omega_0:=\{\omega\in\Omega:\,\mu_n^{(\omega)}\Rightarrow\mu_0\}$.
By assumption $(i)$, $\mathbb{P}(\Omega_0)=1$. For any $\omega\in\Omega_0$, by Skorohod's theorem
(cf. Theorem~1.8 in \cite{EthierKurtz1986}, pages~102--103),
there exists a probability space $(\Omega',\,\mathcal{F}',\,\rho)$ on which $\mathrm{T}$-valued
random elements $Y_n^{(\omega)}$, $n=1,\,2,\,\ldots$, and $Y_0$ are defined such that
$Y_n^{(\omega)}\sim\mu_n^{(\omega)}$, $n=1,\,2,\,\ldots$, $Y_0\sim\mu_0$ and
$d(Y_n^{(\omega)}(\omega'),\,Y_0(\omega'))\rightarrow0$ for $\rho$-almost every $\omega'\in\Omega'$. Let $\Omega_1:=\{\omega\in\Omega:\,(\ref{eq:nullmeasu})\textrm{ holds true}\}$.
Clearly, $\mathbb{P}(\Omega_0\cap\Omega_1)=1$. Fix $\omega\in(\Omega_0\cap\Omega_1)$.
For any $\eta>0$, let
\[
S_{\eta/2}^{(\omega)}:=
\pg{(\tau,\,\zeta)\in K_{\eta/2}:\,\lim_{n\rightarrow\infty}\frac{1}{n}\log\frac{\ptzo}
{\ptnz}(\Data(\omega))
=\textrm{KL}_\infty((\tau_0,\,\zeta_0);\,(\tau,\,\zeta))\quad \textrm{for all }\tau_n\rightarrow \tau}.
\]
By assumptions $(ii)$-$(iii)$, $\Pi(S_{\eta/2}^{(\omega)}|\lambda_0)>0$.
Defined the set $D_{\eta/2}^{(\omega)}:=\{(\omega',\,\zeta):\,(Y_0(\omega'),\,\zeta)\in S_{\eta/2}^{(\omega)}\}$,
\begin{equation}\label{eq:pos}
\int_{\mathrm{Z}}\int_{\Omega'}\mathbf{I}_{D_{\eta/2}^{(\omega)}}(\omega',\,\zeta)\,\mathrm{d}\rho(\omega')\,
\mathrm{d}\tilde{\Pi}(\zeta)=\Pi(S_{\eta/2}^{(\omega)}|\lambda_0)>0.
\end{equation}
By Fubini's theorem, a change of measure and Fatou's lemma,
\begin{eqnarray*}\varliminf_{n\rightarrow\infty}e^{n\eta}D_n
&\geq&\int_{\mathrm{Z}}\int_{\Omega'}\varliminf_{n\rightarrow\infty} \exp\pg{n\pq{\eta-\frac{1}{n}\log\frac{\ptzo}{\pyz}
(\Data(\omega))}}\,\mathrm{d}\rho(\omega')\,\mathrm{d}\tilde{\Pi}(\zeta)\\
&\geq&\int_{\mathrm{Z}}\int_{\Omega'}\mathbf{I}_{D_{\eta/2}^{(\omega)}}(\omega',\,\zeta)\\ &&\hspace*{0.8cm}\times\,\,\varliminf_{n\rightarrow\infty}\exp\pg{n\pq{\eta-\frac{1}{n}\log\frac{\ptzo}{\pyz}
(\Data(\omega))}}\,\mathrm{d}\rho(\omega')\,\mathrm{d}\tilde{\Pi}(\zeta)=\infty,
\end{eqnarray*}
because the integrand is equal to $\infty$ over a set of positive probability, see (\ref{eq:pos}).
Thus, $D_n>e^{-n\eta}$ for all large $n$, $P_0^\infty$-almost surely.
Choosing $0<\eta<c_1\epsilon^2$, for $\delta:=(c_1\epsilon^2-\eta)>0$, we have $\Pi(U_\epsilon^c|\hatl_n,\,\Data)=
N_n/D_n<e^{-n\delta}$ for all large $n$,
$P_0^\infty$-almost surely. The assertion follows.

%%%%%%%%%%%%%%%%%%%%%%%%%%%%%%%%%%%%%%%%%%%%%%%%%%%%%%%%%%%%%%%%%%%%%%%%%%%%%%%%%%%%%%%%%%%%%%%%%%%%%%%%%%%%%%%%

\subsection{Proof of Theorem~\ref{th:strongmerg}} \label{app:th:strongmerg}

We begin by proving \eqref{convergence}. From $(ii)$, for each $\lambda_0\in\tilde \Lambda_0$, $P_0^\infty$-almost surely, $m(\Data|\lambda_0)>0$ for all large $n$. By definition of $\hatl_n$, $0<m(\Data|\lambda_0)\leq\hatm(\Data)<\infty$ for all large $n$,
whence
\begin{equation}\label{eq:disratio}
\frac{\hatm(\Data)}{m(\Data|\lambda_0)}\geq1\qquad\mbox{for all large $n$},
\end{equation}
$P_0^\infty$-almost surely. We prove the reverse inequality. Using $(\mathbf{A1})$, $(i)$ and $(ii)$,
for any $\delta>0$, there exists $\epsilon>0$ (depending on $\delta$, $\theta_0$ and $g(\theta_0)$) so that, with probability greater than or equal to $1-c_2(n\epsilon^2)^{-(1+t)}$,
\begin{equation*}
\begin{split}
\forall\,\lambda\in\Lambda,\quad\frac{m(\Data|\lambda)}{\pto(\Data)}&< e^{-c_1n\epsilon^2 }+\int_{U_\epsilon}R(\ptn)
\pi(\theta|\lambda)\,\mathrm{d}\nu(\theta)\\
& \leq e^{-c_1n\epsilon^2}+\int_{U_\epsilon}R(\ptn)g(\theta)\,\mathrm{d}\nu(\theta)\\
&< e^{-c_1n\epsilon^2}+(1+\delta/3)\int_{U_\epsilon}R(\ptn)
g(\theta_0)\,\mathrm{d}\nu(\theta)\\
&< e^{-c_1n\epsilon^2} + (1+2\delta/3)\int_{U_\epsilon}R(\ptn)
\pi(\theta|\lambda_0)\,\mathrm{d}\nu(\theta),
\end{split}
\end{equation*}
where the second inequality descends from the definition of $g$, because $\pi(\theta|\lambda)\leq g(\theta)$ for all $\theta \in U_\epsilon$, the third one from the positivity and continuity of $g$ at $\theta_0$ and the last one from the fact that $g(\theta_0)=\pi(\theta_0|\lambda_0)$, together with the continuity of $\pi(\theta|\lambda_0)$ at $\theta_0$. By the first Borel-Cantelli lemma, for any $\delta>0$, there exists $\epsilon>0$ so that
\[\forall\,\lambda\in\Lambda,\quad\frac{m(\Data|\lambda)}{\pto(\Data)}<e^{-c_1n\epsilon^2}+ (1+2\delta/3)\int_{U_\epsilon}R(\ptn)
\pi(\theta|\lambda_0)\,\mathrm{d}\nu(\theta)\qquad\mbox{for all large }n,\]
$P_0^\infty$-almost surely. The Kullback-Leibler condition on $\Pi(\cdot|\lambda_0)$ implies that, on a set of
$P_0^\infty$-probability $1$,
\begin{equation}\label{den}
\forall\,a>0,\quad\int_{U_\epsilon}R(\ptn)\pi(\theta|\lambda_0)\,\mathrm{d}\nu(\theta)>e^{-a n}\qquad\mbox{for all large $n$}.
\end{equation}
Therefore, for any $\delta>0$, on a set of $P_0^\infty$-probability $1$, for each
$\lambda\in\Lambda$, $m(\Data|\lambda)\leq (1+\delta)m(\Data|\lambda_0)$ for all large $n$,
which, combined with \eqref{eq:disratio}, proves \eqref{convergence}.

We now prove the convergence of $\hatl_n$. Recall that, by assumption $(\mathbf{A1})$,
for any $\epsilon>0$, on a set of $P_0^\infty$-probability $1$,
\[\forall\,\lambda\in\Lambda,\quad \frac{m(\Data|\lambda)}{\pto(\Data)}<e^{-c_1n\epsilon^2}+\int_{U_\epsilon}R(\ptn)
\pi(\theta|\lambda)\,\mathrm{d}\nu(\theta)\qquad\mbox{ for all large $n$}.\]
For $\delta>0$, define $N_\delta:=\{\lambda\in\Lambda:\,d(\lambda,\,\Lambda_0)\leq\delta\}$.
For any fixed $\delta>0$, by assumption $(iii)$, there exist $\epsilon_1,\,\eta>0$ so that,
on a set of $P_0^\infty $-probability $1$,
\[\sup_{\lambda\in N_{\delta}^c}\frac{m(\Data|\lambda)}{\pto(\Data)}<e^{-c_1n\epsilon_1^2}+(1-\eta)\int_{U_{\epsilon_1}}R(\ptn)
g(\theta)\,\mathrm{d}\nu(\theta)\qquad\mbox{ for all large $n$},\]
whence, using $(i)$ and $(ii)$ on the continuity of $g$ and $\pi(\cdot|\lambda_0)$, $\lambda_0\in \tilde\Lambda_0$, at $\theta_0$,
\[\sup_{\lambda\in N_{\delta}^c}\frac{m(\Data|\lambda)}{\pto(\Data)}<e^{-c_1n\epsilon_1^2}+ (1-\eta/2)\frac{m(\Data|\lambda_0)}{\pto(\Data)}\qquad\mbox{ for all large $n$}.
\]
Using \eqref{den}, we finally get that
$\sup_{\lambda\in N_{\delta}^c} m(\Data|\lambda )<(1-\eta/4)m(\Data|\lambda_0)$ for all large $n$,
$P_0^\infty$-almost surely. The fact that $\eta$ is fixed implies that, for $n$ large enough, $\hatl_n\in N_{\delta}$, a.s.$\,[P_0^\infty]$. Since $\Lambda_0 $ is included in the interior of $\Lambda$, with $P_0^\infty$-probability $1$, $\hatl_n$ belongs to the interior of $\Lambda$ and $\Pi(\cdot|\hatl_n)\ll\nu$ for all large $n$. This fact, combined with consistency of both the EB posterior and $\Pi(\cdot|\lambda_0,\,\Data)$,
and the convergence in \eqref{convergence}, yields that, $P_0^\infty$-almost surely, for any $\epsilon>0$,
\begin{equation*}
\begin{split}
\|\pi(\cdot|\hatl_n,\,\Data)-\pi(\cdot|\lambda_0,\,\Data)\|_1&\leq
\epsilon+\int_{U_\epsilon}\ptn(\Data) \abs{\frac{\pi(\theta|\hatl_n)}{\hatm(\Data)}-\frac{\pi(\theta|\lambda_0)}{m(\Data|\lambda_0)}} \mathrm{d}\nu(\theta)\\
&\leq\epsilon+\abs{\frac{\hatm(\Data)}{m(\Data|\lambda_0)}-1}\\&\qquad\quad+\int_{U_\epsilon} \frac{\ptn(\Data)}{\hatm(\Data)}|\pi(\theta|\hatl_n)-\pi(\theta|\lambda_0)|\,\mathrm{d}\nu(\theta)\\
&\leq 2\epsilon + \int_{U_\epsilon}\frac{\ptn(\Data)}{\hatm(\Data)}|\pi(\theta|\hatl_n)-\pi(\theta|\lambda_0)|\, \mathrm{d}\nu(\theta)
\end{split}
\end{equation*}
for $n$ large enough. We split $U_\epsilon$ into $D_\epsilon:=\{\theta\in U_\epsilon:\,\pi(\theta|\hatl_n)\geq\pi(\theta|\lambda_0)\}$ and
$D_\epsilon^c=\{\theta\in U_\epsilon:\,\pi(\theta|\hatl_n)<\pi(\theta|\lambda_0)\}$.
Since, for any $\delta>0$, if $\epsilon$ is small enough,
$\pi(\theta|\hatl_n)\leq\pi(\theta|\lambda_0)(1+\delta/3)$,
\begin{equation}\label{Despc}
\int_{D_\epsilon}\ptn(\Data)[\pi(\theta|\hatl_n)-\pi(\theta|\lambda_0)]\,\mathrm{d}\nu(\theta)\leq
\frac{\delta}{3}\int_{D_\epsilon}\ptn(\Data)\pi(\theta|\lambda_0)\,\mathrm{d}\nu(\theta)\leq\frac{\delta}{3} \hatm(\Data).
\end{equation}
From consistency of the EB posterior,
\[\int_{U_\epsilon}\ptn(\Data)\pi(\theta|\lambda_0)\,\mathrm{d}\nu(\theta)\leq\hatm(\Data)
<\int_{U_\epsilon}\ptn(\Data)\pi(\theta|\hatl_n)\,\mathrm{d}\nu(\theta)+(\epsilon+\delta/3)\hatm(\Data),\]
whence
\[\int_{D_\epsilon^c}\ptn(\Data)[\pi(\theta|\lambda_0)-\pi(\theta|\hatl_n)]\,\mathrm{d}\nu(\theta)
\leq\int_{D_\epsilon}\ptn(\Data)[\pi(\theta|\hatl_n)-\pi(\theta|\lambda_0)]\,\mathrm{d}\nu(\theta)+
(\epsilon+\delta/3)\hatm(\Data)\]
and, using \eqref{Despc}, \[\int_{D_\epsilon^c}\ptn(\Data)[\pi(\theta|\lambda_0)-\pi(\theta|\hatl_n)]\,\mathrm{d}\nu(\theta)\leq (\epsilon+2\delta/3)\hatm(\Data),\]
which implies that
\[\int_{U_\epsilon} \frac{\ptn(\Data)}{\hatm(\Data)}|\pi(\theta|\hatl_n)-\pi(\theta|\lambda_0)|\, \mathrm{d}\nu(\theta)\leq (\epsilon+\delta)\qquad\mbox{ for all large $n$}.\]
Thus, \eqref{eq:merging} is proved and the proof is complete.

%%%%%%%%%%%%%%%%%%%%%%%%%%%%%%%%%%%%%%%%%%%%%%%%%%%%%%%%%%%%%%%%%%%%%%%%%%%%%%%%%%%%%%%%%%%%%%%%%%%%%%%%%

\subsection{Proof for the non-merging of the posterior in Section \ref{subsec:dege}}\label{app:pr:dege}
We first recall the formal framework.
\begin{thm}\label{th:dege}
Suppose that $\theta_0 \in \Theta_0^c$. Assume that $(\mathbf{A1})$ is satisfied and
\begin{itemize}
\item[$(i)$]there exists $\lambda_0 \in \bar{\Lambda}_0$ such that $\Pi(\cdot |\lambda_0) = \delta_{\theta_0}$,
\item[$(ii)$]with $\Pto$-probability going to $1$, $\hatm(\Data)\geq \pto(\Data)$,
 \item[$(iii)$]the model admits a LAN expansion in the following form: for each $\epsilon>0$, there exists a set, with $\Pto$-probability going to $1$, wherein, uniformly in $\theta\in U_\epsilon$,
  $$ l_n(\theta)-l_n(\hat{\theta}_n) \in - \frac{ n(\theta - \hat{\theta}_n )' I(\theta_0) (\theta - \hat{\theta}_n) }{ 2 }  ( 1 \pm \epsilon),\qquad\mbox{ $\hat{\theta}_n$ denoting the MLE,}$$
\item[$(iv)$]$l_n(\hat{\theta}_n)-l_n(\theta_0)$ converges in distribution to a $\chi^2$-distribution with $k$ degrees of freedom.
\end{itemize}
Then, the EB posterior cannot merge strongly with any Bayes posterior $\Pi(\cdot|\lambda,\,\Data)$, with $\lambda \in \Lambda$ such that the prior density $\pi(\cdot|\lambda)$ is positive and continuous at $\theta_0$.
\end{thm}

\begin{proof}
Define, for any $\delta>0$, the set $\Omega_{n,\,\delta}$ of $\data$'s such that
$e^{l_n(\hat{\theta}_n)-l_n(\theta_0)}\leq 1+\delta$. From assumption $(iv)$, for every $\delta>0$,
$\varliminf_{n\rightarrow\infty}\Pto(\Omega_{n,\,\delta})>0$. From assumption $(ii)$, $\hatm(\Data)/\pto(\Data)\geq 1$. We now study the reverse inequality. Using $(\mathbf{A1})$, for any $\epsilon>0$, on a set $A_n$ with $\Pto$-probability going to $1$,
\[
 \frac{\hatm(\Data)}{\pto(\Data)}=\int_{U_\epsilon}e^{l_n(\theta)-l_n(\theta_0)}\,\mathrm{d}\Pi(\theta| \hatl_n)+O(e^{-n \delta}).
\]
Moreover, using the LAN condition $(iii)$, for every $\theta\in U_\epsilon$,
$$l_n(\theta)-l_n(\theta_0)=l_n(\hat{\theta}_n)-l_n(\theta_0)+\frac{-n(\theta-\hat{\theta}_n)'I(\theta_0)
(\theta-\hat{\theta}_n)}{2}(1+o_p(1)),$$
so that, if $M_n:=M \sqrt{(\log n)/n}$, with $M>0$, on a set of $\Pto$-probability going to $1$,
\[
\int_{\|\theta-\hat{\theta}_n\| > M_n} e^{l_n(\theta)-l_n(\hat{\theta}_n)}\,\mathrm{d}\Pi(\theta|\hatl_n)
= O(n^{-H})\qquad\mbox{ for all $H>0$,}
\]
provided $M$ is large enough. This leads to
\[
\frac{\hatm(\Data)}{\pto(\Data)}=e^{l_n(\hat{\theta}_n)-l_n(\theta_0)}
\int_{U_{M_n}}e^{-n(\theta-\hat{\theta}_n)'I(\theta_0)(\theta-\hat{\theta}_n) /2 }\,\mathrm{d}\Pi(\theta|\hatl_n)+ O(n^{-H}),
\]
where $U_{M_n}:=\{\theta:\,\|\theta-\hat{\theta}_n\|\leq M_n\}$. With abuse of notation, we still denote by $A_n$ the set having $\Pto$-probability going to $1$ wherein the above computations are valid, so that, on $A_n \cap \Omega_{n,\,\delta}$,
\[
\frac{\hatm(\Data)}{\pto(\Data)}\leq 1+2\delta \qquad \mbox{for } n \mbox{ large enough.}
\]
Let $\lambda \in \Lambda$ be such that the prior density $\pi(\cdot|\lambda)$
is positive and continuous at $\theta_0$. Under assumptions $(iii)$ and $(\mathbf{A1})$, usual Laplace expansion of the marginal distribution of $\Data$ yields
\[
\frac{m(\Data|\lambda)}{\pto(\Data)}=\frac{\pi(\theta_0|\lambda)
e^{l_n(\hat{\theta}_n)-l_n(\theta_0)}(2\pi)^{k/2}}{n^{k/2}|I(\theta_0)|^{1/2}}(1+o_p(1)),
\]
so that $m(\Data|\lambda)/\hatm(\Data)=o_p(1)$. We now study the $L_1$-distance between the two posteriors.
If $\Pi(\cdot|\hatl_n)$ is degenerate (\emph{i.e.}, it is not absolutely continuous w.r.t. Lebesgue measure, which plays here the role of $\nu$), then the $L_1$-distance between the EB posterior and the posterior corresponding to $\Pi(\cdot|\lambda)$ is $1$. Thus, we only need to consider the case where $\Pi(\cdot|\hatl_n)$ is absolutely continuous w.r.t. Lebesgue measure. On a set of $\Pto$-probability going to $1$, which we still denote by $A_n$, intersected with $\Omega_{n,\,\delta}$, for each $\theta\in U_{M_n}$,
\begin{equation*}
\begin{split}
\pi(\theta|\hatl_n,\,\Data)-\pi(\theta|\lambda,\,\Data)&=
e^{l_n(\theta)-l_n(\hat{\theta}_n)} \pq{e^{l_n(\hat{\theta}_n)-l_n(\theta_0)}\pi(\theta|\hatl_n)-\frac{ n^{k/2}|I(\theta_0)|^{1/2}}{(2\pi)^{k/2}}+o_p(1)}\\
&= e^{-n(\theta-\hat{\theta}_n)'I(\theta_0)(\theta-\hat{\theta}_n)/2}\frac{n^{k/2}|I(\theta_0)|^{1/2}}{ (2\pi)^{k/2}}(1+o_p(1))\\
&\qquad\qquad\qquad\qquad\times\pq{e^{l_n(\hat{\theta}_n)-l_n(\theta_0)}\pi(\theta|\hatl_n)\frac{(2\pi)^{k/2}}{n^{k/2}|I(\theta_0)|^{1/2}}-1}.
 \end{split}
\end{equation*}
Set $u:=\sqrt{n}I(\theta_0)^{1/2}( \theta - \hat{\theta}_n)$ and define $V_n:=\{u:\,g_n(u)\geq 1-2\delta\}$, where
$$g_n(u):=\pi(\hat{\theta}_n+I(\theta_0)^{-1/2}u/\sqrt{n}|\hatl_n)\frac{(2\pi)^{k/2}}{n^{k/2}|I(\theta_0)|^{1/2}}.$$
To simplify the notation, we also denote by $V_n:=\{\theta=\hat{\theta}_n+I(\theta_0)^{-1/2}u/\sqrt{n}:\,u \in V_n\}$.
Then, for all $c>0$,
\[
\int_{V_n\cap\{\|u\|\leq cM_n\sqrt{n}\}}g_n(u)\,\mathrm{d}u=(2\pi)^{k/2}\int_{V_n\cap\{\|\theta-\hat{\theta}_n\|\leq c M_n\sqrt{n}\}}\pi(\theta|\hatl_n)\,\mathrm{d}\theta\leq (2\pi)^{k/2}
\]
and, by definition of $V_n$,
\[\int_{V_n\cap\{\|u\|\leq cM_n\sqrt{n}\}} g_n(u)\,\mathrm{d}u\geq(1-2\delta)\int_{V_n\cap\{\|u\|\leq cM_n\sqrt{n}\}} \,\mathrm{d}u.\] Hence
\begin{equation}\label{prob}
\int_{V_n\cap\{\|u\|\leq cM_n\sqrt{n}\}}\,\mathrm{d}u\leq(2\pi)^{-k/2}(1-2\delta)^{-1}.
\end{equation}
Note that, on $V_n^c$,
$$\pi(\theta|\hatl_n)\frac{(2\pi)^{k/2}}{n^{k/2}|I(\theta_0)|^{1/2}}< 1-2\delta,$$
so that
$$\pi(\theta|\hatl_n)(1+\delta)\frac{(2\pi)^{k/2}}{n^{k/2}|I(\theta_0)|^{1/2}}-1<-\delta$$
and we can bound from below the $L_1$-distance between the two posteriors:
on $A_n\cap \Omega_{n,\,\delta}$,
\begin{equation*}
\begin{split}
\int_\Theta|\pi(\theta|\hatl_n,\,\Data)-\pi(\theta|\lambda,\,\Data)|\,\mathrm{d}\theta&\geq
\int_{V_n^c\cap U_{M_n}}|\pi(\theta|\hatl_n,\,\Data)-\pi(\theta|\lambda,\,\Data)|\,\mathrm{d}\theta\\
&\geq \delta \int_{V_n^c \cap U_{M_n}}e^{-n(\theta-\hat{\theta}_n)'I(\theta_0)
(\theta-\hat{\theta}_n)/2}\frac{n^{k/2}|I(\theta_0)|^{1/2}}{(2\pi)^{k/2}}\,\mathrm{d}\theta\\
&\geq \delta \int_{V_n^c\cap\{\|u\|\leq cM\sqrt{\log n}\}}\phi(u)\,\mathrm{d}u,
\end{split}
\end{equation*}
for some $c>0$, since $I(\theta_0)$ is positive definite and where $\phi(\cdot)$ is the density of a standard Gaussian distribution on $\R^k$. By choosing $L>0$ large enough and using \eqref{prob},
\begin{equation*}
\begin{split}
\int_{V_n^c\cap\{\|u\|\leq cM\sqrt{\log n}\}}\phi(u)\,\mathrm{d}u&\geq\int_{V_n^c\cap\{\|u\|\leq L\}} \phi(u)\,\mathrm{d}u\\
&\geq\phi(L)\int_{V_n^c\cap\{\|u\|\leq L\}}\,\mathrm{d}u\\
&=\phi(L)\pt{\frac{\pi^{k/2}L^k}{\Gamma(k/2+1)}-\int_{V_n\cap\{\|u\|\leq L\}}\,\mathrm{d}u}\\
&\geq\phi(L)\pt{\frac{\pi^{k/2} L^k}{\Gamma(k/2+1)}-\int_{V_n\cap\{\|u\|\leq cM\sqrt{\log n}\}}\,\mathrm{d}u}\\
&\geq \phi(L)\frac{\pi^{k/2}L^k}{2\Gamma(k/2+1)}>0,
\end{split}
\end{equation*}
which completes the proof.
\end{proof}

%%%%%%%%%%%%%%%%%%%%%%%%%%%%%%%%%%%%%%%%%%%%%%%%%%%%%%%%%%%%%%%%%%%%%%%%%%%%%%%%%%%%%%%%%%%%%%%%%%%%%%%%%

\end{document}